\documentclass[onefignum,onetabnum, reqno]{siamonline190516} % delete reqno to show number on the left
\usepackage{amsfonts, amsmath, amssymb}
\usepackage{graphbox,graphicx}
\usepackage{epstopdf}
\usepackage[algo2e,linesnumbered,boxed]{algorithm2e}
\usepackage{dsfont}
\usepackage{booktabs}
\usepackage{todonotes}
\usepackage{mathtools}
\usepackage{tikz}
\usepackage{wrapfig}

\ifpdf
  \DeclareGraphicsExtensions{.eps,.pdf,.png,.jpg}
\else
  \DeclareGraphicsExtensions{.eps}
\fi

\usepackage{enumitem}
\setlist[enumerate]{leftmargin=.5in}
\setlist[itemize]{leftmargin=.5in}

\newsiamremark{remark}{Remark}
\newsiamthm{assumption}{Assumption}
\crefname{assumption}{Assumption}{Assumptions}
\newsiamremark{hypothesis}{Hypothesis}
\crefname{hypothesis}{Hypothesis}{Hypotheses}
\newsiamthm{claim}{Claim}
\newsiamthm{scenario}{Scenario}
\crefname{scenario}{Scenario}{Scenarios}
\crefname{section}{Section}{Sections}
\crefname{subsection}{Section}{Sections}

\headers{Risk-Averse Markov Decision Processes through a Distributional Lens}{Cheng \& Jaimungal}

\title{Risk-Averse Markov Decision Processes through a Distributional Lens \thanks{
\funding{SJ  would like to acknowledge support from the Natural Sciences and Engineering Research Council of Canada (grants RGPIN-2018-05705 and RGPAS-2018-522715).}}}

\author{
Ziteng Cheng
\thanks{Department of Statistical Sciences, University of Toronto
  (\email{sebastian.jaimungal@utoronto.ca}, \url{http://sebastian.statistics.utoronto.ca}, \email{ziteng.cheng@utoronto.ca}).}
 \and
Sebastian Jaimungal
\footnotemark[2]
}

\usepackage{amsopn}

\tikzstyle{reward}=[shape=circle,draw=blue!50,fill=blue!10]
\tikzstyle{action}=[shape=circle,draw=green,fill=green!10]
\tikzstyle{state}=[shape=circle,draw=red!50,fill=red!10]
\tikzstyle{gru}=[shape=rectangle,draw=black!50,fill=lime!10]
\tikzstyle{obs}=[shape=circle,draw=blue!50,fill=blue!10]
\tikzstyle{lightedge}=[<-,dotted]
\tikzstyle{mainstate}=[state,thick]
\tikzstyle{mainedge}=[<-,thick]

\newcounter{example}[section]

\usepackage{bm}
\usepackage[normalem]{ulem} % for \sout{}
\usepackage{bbm} % for \1 defined below
\usepackage{kpfonts}
\usepackage{cancel}

\newcommand{\wt}{\widetilde}

\newcommand{\1}{\mathbbm{1}}            % preferable way of writing indicator function
\newcommand{\set}[1]{\{#1\}}            % set: {xyz} to be used for inline formulas
\DeclareMathOperator{\dif}{d \!}        % used for differential, same as in commath.sty
\DeclareMathOperator*{\argmin}{arg\,min} % argmin
\DeclareMathOperator\cvar{CV@R}
\DeclareMathOperator\avar{AV@R}
\DeclareMathOperator\supp{support}

\def\cA{\mathcal{A}}
\def\cB{\mathcal{B}}

\def\cE{\mathcal{E}}

\def\cL{\mathcal{L}}
\def\cM{\mathcal{M}}
\def\cN{\mathcal{N}}

\def\cP{\mathcal{P}}

\def\bA{\mathbb{A}}

\def\bD{\mathbb{D}}
\def\bE{\mathbb{E}}
\def\bF{\mathbb{F}}
\def\bG{\mathbb{G}}

\def\bM{\mathbb{M}}
\def\bN{\mathbb{N}}

\def\bP{\mathbb{P}}

\def\bR{\mathbb{R}}
\def\bS{\mathbb{S}}

\def\bU{\mathbb{U}}
\def\bV{\mathbb{V}}

\def\bX{\mathbb{X}}
\def\bY{\mathbb{Y}}
\def\bZ{\mathbb{Z}}

\usepackage{mathrsfs}
\def\sA{\mathscr{A}}

\def\sD{\mathscr{D}}

\def\sF{\mathscr{F}}
\def\sG{\mathscr{G}}
\def\sH{\mathscr{H}}

\def\sU{\mathscr{U}}

\def\sX{\mathscr{X}}
\def\sY{\mathscr{Y}}

\def\fA{\mathfrak{A}}

\def\fC{\mathfrak{C}}

\def\fH{\mathfrak{H}}

\def\fX{\mathfrak{X}}

\def\fZ{\mathfrak{Z}}

\def\fp{\mathfrak{p}}

\def\fv{\mathfrak{v}}

\def\kcol{\breve{k}}
\def\kcaution{\hat{k}}
\usepackage{todonotes}
\usepackage{makecell}

\graphicspath{{./Figures/}}

\begin{document}

\maketitle

% REQUIRED
\begin{abstract}

By adopting a distributional viewpoint on law-invariant convex risk measures, we construct dynamics risk measures (DRMs) at the distributional level. We then apply these DRMs to investigate Markov decision processes, incorporating latent costs, random actions, and weakly continuous transition kernels. Furthermore, the proposed DRMs allow risk aversion to change dynamically. Under mild assumptions, we derive a dynamic programming principle and show the existence of an optimal policy in both finite and infinite time horizons. Moreover, we provide a sufficient condition for the optimality of deterministic actions. For illustration, we conclude the paper with examples from optimal liquidation with limit order books and autonomous driving.
\end{abstract}

% REQUIRED

% REQUIRED

\section{Introduction}

Markov Decision Processes (MDPs) may be viewed as discrete-time stochastic control problems for sequential decision making in situations where costs are partly random and partly under the control of a decision maker. Classical MDP theory is concerned with minimizing the expected discounted total cost and, in many cases, the minimization problem is solved by establishing a Dynamic Programming Principle (DPP). Results in the vast area of MDPs may be founded in several textbooks, e.g., \cite{Bertsekas1996book,HernandezLerma1996book,Bauerle2011book}. The classical expected performance criteria is, however, limited in its application and, in many cases, it is prudent to incorporate risk assessment into decision making. 

One popular criterion is based on convex risk measures \cite{Artzner1999Coherent,Delbaen2002Coherent, Folmer2002Convex}. A na\"ive combination of convex risk measures and discounted total costs, however, lacks time consistency, hindering the derivation of a corresponding DPP. Roughly speaking, time consistency refers to the property that smaller scores in future epochs guarantee a smaller score in the current epoch. We refer to \cite{Bielecki2017Survey} for a survey on various definitions of time consistency. There is a stream of literature (see, e.g., \cite{Detlefsen2005Conditional, Gianin2006Risk, Riedel2007Dynamic, Ruszczynski2010Risk, Filipovic2012Approaches, Pflug2016Time, Chow2015Risk, Bauerle2021Minimizing}) that studies time consistency from multiple angles and/or attempts to integrate convex risk measures and their variations into MDPs. While here we are not concerned with model uncertainty, we would like to point out \cite{Bielecki2021Risk} and the references therein for a framework that handles model uncertainty in MDPs.

Here, we focus on the risk-averse MDPs framework proposed in \cite{Ruszczynski2010Risk}. While a plethora of studies exist on conditional risk measures (CRMs) and dynamic risk measures (DRMs), the problem of risk-averse MDPs and their corresponding DPP cannot be straightforwardly inferred from the established properties of these risk measures. Broadly speaking, DPPs transform a sequential optimization problem into the task of solving an operator Bellman equation. In the context of risk-averse MDPs, this demands separate research, as in general, a readily solvable DPP in terms of operator equations requires additional conditions beyond properties of CRMs and DRMs. This is particularly true when considering the attainability of optimal actions. Below, we offer a concise review of the established DPPs that align with the general framework initiated by \cite{Ruszczynski2010Risk}. \cite{Ruszczynski2010Risk} considers deterministic costs, introduces the notion of risk transition mappings, and uses them to construct, in a recursive manner,  a class of (discounted) DRMs. The author proceeds to derive both finite and infinite (with bounded costs) time horizon DPPs for such DRMs. We also refer to \cite{Ruszczynski2014Erratum} for the assumptions needed. \cite{Shen2014Risk} extends the infinite horizon DPP to unbounded costs as well as for average DRMs. The risk transition mappings involved are assumed to exhibit an analogue of a strong Feller property. \cite{Chu2014Markov} studies a similar infinite horizon DPP with unbounded costs under a different set of assumptions. Recently, \cite{Bauerle2021Markov} considers unbounded latent costs and establishes the corresponding finite and infinite horizon DPPs. They also prove sufficiency of Markovian actions against history dependent actions. They construct DRMs, for finite time horizon problems, from iterations of static risk measures that are Fatou and law invariant. The infinite horizon problems require in addition the coherency property. They also require the underlying MDP to exhibit a certain path-wise continuous/semi-continuous transition mechanism.  \cite{coache2023reinforcement,coache2023conditionally} develops a computational approach for optimization with dynamic convex risk measures using deep learning techniques.  Finally, it is noteworthy that the concept of risk form is introduced in \cite{Dentcheva2020Risk} and is applied to handle two-stage MDP with partial information and decision-dependent observation distribution. 

The main goal of this paper is to study finite and infinite horizon risk averse MDPs in a similar framework as above, but with latent costs and randomized actions, under a weakly continuous transition mechanism, subject to state-dependent DRMs. Below we provide a high level discussion on some of the open problems we address in this paper. We would like to note that in some of the earlier risk-averse MDP frameworks, latent costs were not considered. We believe that in many circumstances, latent costs can be used to account for the risk related to factors such as time-discretization and processing delays. To the best of our knowledge, \cite{Bauerle2021Markov} is the first to consider latent costs within risk-averse MDP frameworks. Their DRMs, however, are constructed from compositions of static risk measures. A framework that allows for state-dependent risk measures allows for greater flexibility in adjusting the level of risk aversion depending on where one is in state space (e.g., in the context of portfolio allocation, an investor may become more risk averse as their wealth increase). Apart from the points mentioned above, it is worth questioning whether it is possible to modify the existing risk-averse MDP framework to account for the randomness in randomized actions in a risk-averse manner. We note that randomized actions are allowed in \cite{Chu2014Markov}, but are evaluated in an arguably risk-neutral manner. Lastly, an adequate discussion on weakly continuous transition kernels is missing from the existing literature on risk averse MDPs. Weakly continuous transition kernels are often preferred in practice, as they provide the flexibility to consider transition dynamics that are partly random and partly deterministic. Moreover, they facilitate data-based modeling by removing the need of working with the density of the kernel or the underlying probability spaces, which typically requires additional model assumptions. 

Our approach to formulating risk-averse MDPs is grounded in the understanding that law-invariant convex risk measures can be interpreted as functionals defined on the space of probabilities over $\bR$. This perspective has been effectively employed in various contexts to further the development of risk measure theory (cf. \cite{Weber2006Distribution, Acciaio2013Are, Delage2019Dicesion, Fadina2023Framework}). Notably, \cite{Weber2006Distribution} is a seminal contribution that systematically investigates static risk measures and DRMs from a distributional standpoint. In a similar vein, we explore DRMs at the level of distributions, conceptualizing them as nested compositions of state-dependent law-invariant convex risk measures.  It is important to emphasize that previous studies on DRMs at the distributional level, such as in \cite{Weber2006Distribution, Bauerle2021Markov}, primarily focused on static one-step risk measures in characterization or construction. Our approach, which allows the risk measures to vary according to the state, introduces additional complexity. A significant advantage of this distributional-level construction is that it automatically ensures MDPs with identical distributions are treated as equivalent in terms of risk when assessed under the proposed DRMs. Moreover, it seamlessly integrates latent costs and random actions through the concept of regular conditional distributions. Furthermore, our framework provides an appropriate foundation for balancing various assumptions, including a weakly continuous transition kernels, while still ensuring the attainment of the optimal outcome. This, in turn, allows for greater flexibility in risk-averse modeling. It's important to note that, although the construction above may seem like a straightforward alteration of existing frameworks, it involves some unique technical aspects that have not been previously discussed. %This is evidenced by the absence of a weakly continuous transition kernel in extant risk-averse MDPs literature.
For simplicity, we consider bounded costs, which allows for conditional risk mappings that contain essential supremum as a major ingredient -- a feature that is often omitted otherwise.

The main contributions of this paper can be summarized as follows:
\begin{enumerate}
\item We construct DRM at a distributional level, by compounding state-dependent law invariant convex risk measures in both finite and infinite horizon. The infinite horizon definition requires that the law-invariant convex risk measures involved are normalized. The regularities underpinning of these constructions are examined in detail in \cref{subsec:DDRM}.

\item Under mild assumptions, we derive both finite and infinite horizon DPPs, for MDPs with latent costs and randomized actions as well as weakly continuous transition kernel, subject to the aforementioned DRMs. We prove the existence of an optimal policy that is Markovian. Furthermore, we demonstrate that, in both finite and infinite horizons, the optimal Markovian policy is no worse than any other history-dependent policies.  We refer to \cref{thm:FiniteDPP}, \cref{prop:MarkovControlOptFinite} and \cref{thm:InfiniteDPP}, \cref{prop:MarkovControlOpt} for detailed statements.

\item In \cref{prop:wpSingleton}, we provide a sufficient condition for the optimality of deterministic actions. This condition could be useful in scenarios where the agent prefers to avoid the complexity and unpredictability associated with randomized actions. 
\end{enumerate}

To demonstrate the practical application of the proposed framework of risk-averse MDPs, we also present two examples in \cref{sec:Examples} on optimal liquidation with limit order books and autonomous driving. The optimal liquidation example, posed within a finite horizon context with discrete state and action spaces, underscores the potential of the framework to incorporate latent costs and adjust risk aversion in response to the evolving state. The autonomous driving example, posed within an infinite horizon context, showcases the framework's capability to handle a large and complex state space, a requirement often necessary for constructing more realistic models. It's worth noting that the autonomous driving example naturally necessitates certain degenerate transition mechanisms. While these are not compatible with strongly continuous transition kernels in general, they can be easily incorporated into a weakly continuous transition kernel as we do here.

The remainder of the paper is structured as follows. In \cref{sec:Setup}, we first introduce our notation, then recall definitions and basic properties of various important concepts, and establish preliminary results, such as the construction of DRMs at the level of distributions. Formulations for risk-averse MDP are organized in the end of the section.  \cref{sec:Aux} is devoted to auxiliary results. We introduce some useful operators related to Markovian policy and investigate their regularities. In \cref{sec:DPPFinite}, we derive the finite horizon DPP for Markovian actions and argue that Markovian actions can achieve the optimal. The analogous results for infinite horizon are presented in \cref{sec:DPPInfinite}. In \cref{sec:SuffDeter}, we establish a sufficient condition on the optimality of deterministic actions. \cref{sec:Examples} contains a variety of examples that serve to illustrate the concepts proposed in this study. We also provide an example illustrating the necessity of randomized actions and some technical lemmas in \cref{app:Example} and \cref{app:Lemmas}, respectively. Finally, for reference, \cref{app:notations} contains a glossary of notation.

\section{Setup and preliminaries}\label{sec:Setup}

To formulate our problem, we first specify the spaces related to the underlying process, action process, and probability, among others. We use the following notations for various spaces throughout the paper. We also provide in Appendix \ref{app:notations} a glossary of notations used throughout the paper, but introduce them as they arise.

\begin{itemize}

\item[-] We write $\bN:=\set{1,2,...}$ and $\bN_0:=\set{0}\cup\bN$. We let $\bR$ denote the real line. We endow $\bR$ with Borel $\sigma$-algebra $\cB(\bR)$. We let $\overline\bR:=\bR\cup\set{+\infty}\cup\set{-\infty}$ and $\cB(\overline\bR):=\sigma(\cB(\bR)\cup\set{\set{+\infty},\set{-\infty}})$.

\item[-] For any measurable space $(\bY,\sY)$, we write $\ell^\infty(\bY,\sY)$ for the set of bounded real-valued $\sY$-$\cB(\bR)$ measurable functions. We let $\ell^\infty(\bN;\bY,\cB(\bY))$
denote the set of $\fv = (v_t)_{t\in\bN}\subseteq \ell^\infty(\bY,\cB(\bY))$, and equip it with norm $\|\fv\|_\infty:=\sup_{t\in\bN, y\in\bY}|v_t(y)|$, which makes\\ $\ell^\infty(\bN;\bY,\cB(\bY))$ complete. For any $y\in\bY$, the Dirac probability measure at $y$, denoted by $\delta_y$, is defined as $\delta_y(A) := \1_{A}(y)$ for $A\in\sY$.

\item[-] Let $(\Omega, \sH, \bP)$ be a complete probability space. We suppose $(\Omega, \sH, \bP)$ does not have atoms. We write $L^\infty(\Omega,\sH,\bP)$ for the set of real-valued $\sH$-$\cB(\bR)$ random variables that are $\bP$-almost surely bounded. For $Z\in L^\infty(\Omega,\sH,\bP)$, we define $\|Z\|_\infty:=\inf\set{r\in\bR:\bP(|Z|>r)=0}$.

\item[-] We let $\bF:=(\sF_t)_{t\in\bN}$ and $\bG:=(\sG_t)_{t\in\bN}$ be filtrations of $\sH$ such that $\sF_t\subseteq\sG_t$ for all $t\in\bN$ and $\sF_1$ contains all $\bP$-negligible sets. We set $\sF_0:=\sG_0:=\set{\emptyset,\Omega}$. Later in Section \ref{subsec:controlledP}, we associate $\bF$ with the state process and $\bG$ with the state-action process. We also define $\sU_t:=\sG_{t-1}\vee\sF_t$ for $t\in\bN$. It follows that $\bU:=(\sU_t)_{t\in\bN}$ is also a filtration. We also set $\sU_0:=\set{\emptyset,\Omega}$. This $\bU$ is used to represent the information available at times of decision making.

\item[-]  Let $\bX$ be a complete separable metric space equipped with Borel $\sigma$-algebra $\cB(\bX)$, which models the state space. Let $\Xi$ be the set of probability measures on $\cB(\bX)$. We endow $\Xi$ with the weak topology, which is the coarsest topology on $\Xi$ containing sets $\left\{\xi\in\Xi: \int_{\bX}f(a)\xi(\dif x)\in U\right\}$ with $f\in C_b(\bX)$ and $U\subseteq\bR$ open. The corresponding Borel $\sigma$-algebra is denoted by $\cB(\Xi)$. The evaluation $\sigma$-algebra on $\Xi$, denoted by $\cE(\Xi)$, 
is the $\sigma$-algebra containing sets of the form $\set{\xi\in\Xi:\xi(A)\in B}$ with $A\in\cB(\bX)$ and $B\in\cB([0,1])$. In other words, $\cE(\Xi)$ is the smallest $\sigma$-algebra on $\Xi$ such that the mapping $A\mapsto\xi(A)$ is $\cE(\Xi)$-$\cB([0,1])$ measurable for any $A\in\cB(\bX)$. Equivalently,\footnote{One direction is obvious because $\xi(A)=\int_{\bX}\1_A(x)\mu(\dif x)$ for $A\in\cB(\bX)$. The other direction follows from an application of monotone class theorem for functions (cf. \cite[Theorem 5.2.2]{Durrett2019book}) on $\big\{f\in\ell^\infty(\bX,\cB(\bX): \xi\mapsto\int_{\bX} f(x)\xi(\dif x) \text{ is } \cE(\Xi)\text{-}\cB(\bR) \text{ measurable } \big\}$, where $\cE(\Xi)$ is defined in the initial way, along with the fact that pointwise convergence preserves measurability (cf. \cite[Section 4.6, Lemma 4.29]{Aliprantis2006book}).}  $\cE(\Xi)$ can be defined as the $\sigma$-algebra containing sets $\set{\xi\in\Xi:\int_{\bX}f(x)\xi(\dif x) \in B}\,$ for all $f\in\ell^\infty(\bX,\cB(\bX))$ and $B\in\cB(\bR)$. 
In view of Lemma \ref{lem:sigmaAlgBE}, we have $\cB(\Xi)=\cE(\Xi)$. 

\item[-]   Let $\bA$ be another complete separable metric space equipped with Borel $\sigma$-algebra $\cB(\bA)$, which models the action space. Let $\Lambda$ be the set of probability measures on $\cB(\bA)$. We endow $\Lambda$ with the weak topology and the corresponding Borel $\sigma$-algebra $\cB(\Lambda)$. The evaluation $\sigma$-algebra on $\Lambda$ is denoted by $\cE(\Lambda)$. By Lemma \ref{lem:sigmaAlgBE} again, we have $\cB(\Lambda)=\cE(\Lambda)$.

\item[-] We consider policies subject to state dependent admissible domains of actions. For $t\in\bN$, we let $\cA_t:(\bX,\cB(\bX))\to 2^{\bA}$ be nonempty closed valued and weakly measurable, i.e., $\set{x\in\bX:\cA_t(x)\cap U\neq\emptyset}\in\cB(\bX)$ for any open $U\subseteq\bA$. For each $t\in\bN$ and $x\in\bX$, we let $\varpi_t(x)$ be the set of probability measures on $\cB(\bA)$ such that $\pi(\cA_t(x))=1$ for any $\pi\in\varpi_t(x)$, and let $\Pi_t$ consist of $\pi_t:(\bX,\cB(\bX))\to(\Lambda,\cE(\Lambda))$ such that $\pi_t(x)\in\varpi_t(x)$ for $x\in\bX$.\footnote{$\Pi_t$ is not empty. To see this, note by Kuratowski and Ryll-Nardzewski measurable selection theorem (cf. \cite[Theorem 18.13]{Aliprantis2006book}), there exists $\alpha:(\bX,\cB(\bX))\to(\bA,\cB(\bA))$ such that $\alpha(x)\in\cA_t(x)$ for $x\in\bX$. It follows that if $\pi(x):=\delta_{\alpha(x)}$, then $\pi\in\Pi_t$.} $\Pi$ is the set of $\fp:=(\pi_t)_{t\in\bN}$ such that $\pi_t\in\Pi_t$ for all $t\in\bN$.

\item[-] For a closed $F\subseteq\bR$, we let $\cP(F)$ the set of probability measures on $(F,\cB(F))$, endowed with weak topology. The corresponding Borel $\sigma$-algebra is denoted by $\cB(\cP(F))$. Since $\cP(F)$ is separable and metrizable (cf. \cite[Section 15.3, Theorem 15.15]{Aliprantis2006book}), invoking Lemma \ref{lem:sigmaAlgBE} again, $\cB(\cP(F))$ coincides with the evaluation $\sigma$-algebra $\cE(\cP(F))$. When $F\subset\bR$, we may treat $\mu\in\cP(F)$ as an element of $\cP(\bR)$ with $\supp \mu\subseteq F$. We let $\cP_b(\bR)$ consists of probability measures on $(\bR,\cB(\bR))$ with bounded support, i.e., $\cP_b(\bR)=\bigcup_{n\in\bN}\cP([-n,n])$. We also define $\cP_b(\bR^2)$ in a similar manner.

\end{itemize}

\subsection{Regular Conditional Distribution}\label{subsec:RegCondDist}

In this section, we recall the definition of regular conditional distributions, as it plays a crucial role in many aspects of the paper.

Consider $Y:(\Omega,\sH)\to(\bY,\sY)$ and  $\sG\subseteq\sH$. The conditional distribution of $Y$ given $\sG$, defined as $\set{\bP(Y\in B|\sG):=\bE(\1_B(Y)|\sG)}_{B\in\sY}$, can be viewed as a function of $\omega\in\Omega$ and $B\in\sY$. For each $B\in\sY$, however, $\bP(Y\in B|\sG)$ is defined only almost surely, which hinders us from using $\bP(Y\in \,\cdot\,|\sG)$ as a probability measure depending on $\omega\in\Omega$ (countable additivity may not hold). To resolve such issues, we recall the notion of a regular conditional distribution.
\begin{definition}\label{def:RegCondDist}
$P^{Y|\sG}:\Omega\times\sY\to[0,1]$ is a regular version of $\bP(Y\in\,\cdot\,|\sG)$ if 
\begin{itemize}
\item[(i)] for each $A\in\sY$, $\omega\mapsto P^{Y|\sG}(\omega,A)$ is $\sG$-$\cB(\bR)$ measurable;
\item[(ii)] for each $\omega\in\Omega$, $P^{Y|\sG}(\omega,\,\cdot\,)$ is a probability measure on $\sY$;
\item[(iii)] for each $A\in\sY$, $P^{Y|\sG}(\omega, A) = \bP(Y\in A|\sG)(\omega)$ for $\bP$-a.e. $\omega\in\Omega$.	
\end{itemize}
\end{definition}
If $\sG$ is the $\sigma$-algebra generated by a random variable, say $X$, we  write $P^{Y|X}$ instead of $P^{Y|\sigma(X)}$.

Let the set of probability measures on $\sY$ be denoted by $\cP$ and endowed with $\sigma$-algebra $\cE(\cP)$. Because for $A\in\sY,\,B\in\cB(\bR)$,
\begin{align*}
\set{\omega\in\Omega:P^{Y|\sG}(\omega,\,\cdot\,)\in\set{\zeta\in\cP:\eta(A) \in B} } = \set{\omega\in\Omega:P^{Y|\sG}(\omega,A)\in B} \in \sG,
\end{align*}
by \cite[Section 4.5, Corollary 4.24]{Aliprantis2006book}, $\omega\mapsto P^{Y|\sG}(\omega,\,\cdot\,)$ is a measure-valued $\sG$-$\cE(\cP)$ random variable.

By \cite[Chapter I Section 3, Theorem 3]{Gikhman1974book} (see also \cite[Theorem 10.4.8 and Example 10.4.9]{Bogachev2007book}), if $\bY$ is a complete separable metric space and $\sY=\cB(\bY)$ is the corresponding Borel $\sigma$-algebra, then $\bP(Y\in\,\cdot\,|\sG)$ always has a regular version. Moreover, $P^{Y|\sG}(\omega,\,\cdot\,)$ as a probability measure is unique up to a $\bP$-negligible set of $\omega\in\Omega$ (cf. \cite[Lemma 10.4.3]{Bogachev2007book}). In view of Lemma \ref{lem:sigmaAlgBE}, $P^{Y|\sG}$ is also $\sG$-$\cB(\cP)$ measurable.

For a nonnegative $\sY$-$\cB(\bR)$ measurable $f$, for each $\omega\in\Omega$, we consider
the Lebesgue integral $\int_{\sY}f(y)P^{Y|\sG}(\omega,\dif y)$. When no confusion arises, we omit $\omega$ and write $P^{Y|\sG}(B)$ and $\int_{\sY}f(y)P^{Y|\sG}(\dif y)$ instead. Clearly, for any $A\in\sY$,
\begin{align}\label{eq:RegCondProbInt}
\int_{\sY}\1_A(y)P^{Y|\sG}(\dif y) = P^{Y|\sG}(A) = \bP(Y\in A|\sG),\quad\bP-a.s.\;.
\end{align}

\subsection{Controlled process $(\fX,\fA)$}\label{subsec:controlledP}
We let $\fX:=(X_t)_{t\in\bN}$ be an $\bX$-valued $\bF$-adapted process, i.e., $X_t$ is $\sF_t$-$\cB(\bX)$ measurable for all $t\in\bN$. Clearly, $\fX$ is a $\bG$-adapted process as well. We also let $\fA:=(A_t)_{t\in\bN}$ be an $\bA$-valued $\bG$-adapted process. $\fX$ and $\fA$ represent the underlying state process and the action process, respectively. Heuristically, letting $\fA$ be $\bG$-adapted allows us to have randomized actions with the property $A_t\sim\pi_t(X_t)$, where $\pi_t:\bX\to\Lambda$. Below we recall the concepts of a transition kernel from $(X_t,A_t)$ to $X_{t+1}$ and Markovian action.

Suppose that for $t\in\bN$, we have
\begin{align}\label{eq:GMarkov}
\bP(X_{t+1}\in B\,|\,\sG_t) = \bP(X_{t+1}\in B\,\big|\,\sigma(X_t)\vee\sigma(A_t)),\quad B\in\cB(\bX).
\end{align}
Let $P^{X_{t+1}|(X_t,A_t)}$
be the corresponding regular conditional distribution. By Definition \ref{def:RegCondDist} (i), for each $B\in \cB(\bX)$, $\omega\mapsto P^{X_{t+1}|(X_t,A_t)}(\omega,B)$ is $\sigma(X_t)\vee\sigma(A_t)$-$\cB(\bR)$ measurable. It follows from \cite[Section 4.8, Theorem 4.11]{Aliprantis2006book} that there is a mapping $h^t_B:(\bX\times\bA,\cB(\bX)\otimes\cB(\bA))\to(\bR,\cB(\bR))$ such that $h^t_B(X_t(\omega),A_t(\omega)) = P^{X_{t+1}|(X_t,A_t)}(\omega,B)$ for all $\omega\in\Omega$. Then by Definition \ref{def:RegCondDist} (ii), for $(x,a)\in\bX\times\bA$ such that $(X_t(\omega),A_t(\omega))=(x,a)$ for some $\omega\in\Omega$, we have $h^t_\cdot(x,a)$
is a probability measure on $\cB(\bX)$. For $(x,a)$ that does not belong to the pointwise range of $(X_t,A_t)$, we may set $h_B(x,a)=\delta_{x_0}(B)$ for some $x_0\in\bX$ so that $h_\cdot(x,a)$ is a probability measure on $\cB(\bX)$ and $h_B$ is still measurable. By writing $P(t,x,a,B)=h^t_B(x,a)$, we have
\begin{align}\label{eq:transkernel}
\bP(X_{t+1}\in B\,\big|\,\sigma(X_t)\vee\sigma(A_t)) = P(t,X_t, A_t, B) ,\quad\bP-a.s.,\,t\in\bN,\,B\in \cB(\bX).
\end{align}
Note that for any $A\in\cB(\bX)$ and $B\in\cB(\bR)$ we have
\begin{align*}
&\left\{(x,a)\in\bX\times\bA:P(t,x,a,\,\cdot\,)\in\left\{\xi\in\Xi:\xi(A)\in B\right\} \right\} \\
&\quad = \left\{(x,a)\in\bX\times\bA:P(t,x,a,A) \in B \right\} = \left\{(x,a)\in\bX\times\bA:h^t_A(x,a) \in B \right\} \in \cB(\bX)\otimes\cB(\bA),
\end{align*}
thus by \cite[Section 4.5, Corollary 4.24]{Aliprantis2006book}, $(x,a)\mapsto P(t,x,a,\cdot)$ is $\cB(\bX)\otimes\cB(\bA)$-$\cE(\Xi)$ measurable.

Finally, we say $\fA$ is a Markovian action if 
\begin{align}\label{eq:MarkovControl}
\bP(A_t\in B\,|\,\sU_t) = \bP(A_t\in B\,|\,\sigma(X_t)),\quad B\in\cB(\bA),\,t\in\bN.
\end{align}
Then, with similar reasoning as before, for any $t\in\bN$ there is a mapping $\pi_t:(\bX,\cB(\bX))\to(\Lambda,\cE(\Lambda))$ such that $\omega\mapsto\pi_t(X_t(\omega))$ is a regular version of $\bP(A_t\in \,\cdot\,|\,\sigma(X_t))$, and 
\begin{align}\label{eq:MarkovControlkernel}
\bP(A_t\in B\,|\,\sigma(X_t)) = [\pi_t(X_t)](B),\quad\bP-a.s.,\,B\in\cB(\bA).
\end{align} 

\subsection{Law-invariant convex risk measures}
Let us first recall the definition of a (proper) convex risk measure on $L^\infty(\Omega,\sH,\bP)$. 
\begin{definition}
$\rho:L^\infty(\Omega,\sH,\bP)\to(-\infty,\infty]$ is a convex risk measure if 
\begin{itemize}
\item[(i)] (Translation invariance) $\rho(X+c)=\rho(X)+c$ for any $Z\in L^\infty(\Omega,\sH,\bP)$ and $c\in\bR$;
\item[(ii)] (Motonicity) $\rho(Z)\le \rho(Z')$ for any $Z,Z'\in L^\infty(\Omega,\sH,\bP)$ with $Z\le Z'$ $\bP$-a.s.;  
\item[(iii)] (Convexity) $\rho(\theta Z+(1-\theta)Z')\le \theta\rho(Z)+(1-\theta)\rho(Z')$ for any $Z,Z'\in L^\infty(\Omega,\sH,\bP)$ and $\theta\in[0,1]$.
\end{itemize}
We say $\rho$ is normalized if $\rho(0)=0$, and $\rho$ is law-invariant if $\cL(Z)=\cL(Z')$ implies $\rho(Z)=\rho(Z')$, where $\cL(Z)$ stands for the law of $Z$.
\end{definition}

We can employ a law-invariant convex risk measure $\rho$ to define a functional on $\cP_b(\bR)$, by using a quantile transform. More precisely, we consider a mapping, denoted by $\sigma$, from the $\cP_b(\bR)$ to $(-\infty,\infty]$ defined by $\sigma(\mu):=\rho(F^{-1}_\mu(U))$, where $F_\mu(r):=\mu((-\infty,r])$, $F^{-1}_\mu(u):=\inf\set{r\in\bR: F(r)\ge u}$, and $U\sim\text{Unif}(0,1)$ in $L^\infty(\Omega,\sH,\bP)$. Adopting a distributional perspective can sometimes be beneficial in advancing the theoretical development of risk measures; see, e.g. \cite{Weber2006Distribution, Acciaio2013Are, Delage2019Dicesion, Fadina2023Framework}. For the remainder of this paper, we  focus on this distributional viewpoint and to facilitate this viewpoint, we introduce the following definition.
\begin{definition}\label{def:DCRM}
$\sigma:\cP_b(\bR)\to(-\infty,\infty]$ is a law-invariant convex risk measure (viewed at the level of distributions) if it satisfies 
\begin{itemize}
\item[(i)] $\sigma(\mu*\delta_c)=\sigma(\mu)+c$ for any $c\in\bR$.
\item[(ii)] For any $\mu,\mu'\in\cP_b(\bR)$, $\mu((-\infty,r])\ge \mu'((-\infty,r])$ for all $r\in\bR$ implies $\sigma(\mu)\le\sigma(\mu')$.
\item[(iii)] For any $\zeta\in\cP_b(\bR^2)$ with bounded support, denote the first and second marginal of $\zeta$ by $\zeta_0$ and $\zeta_1$, respectively. For $\theta\in[0,1]$, let $\zeta_\theta\in\cP_b(\bR)$ be characterized by 
$$\zeta_\theta((-\infty,r])=\zeta(\set{(r_0,r_1)\in\bR^2: (1-\theta)r_0+\theta r_1\le r}).$$
Then, $\sigma(\zeta_\theta)\le \theta\,\sigma(\zeta_0) + (1-\theta)\,\sigma(\zeta_1)$.
\end{itemize}
We say $\sigma$ is normalized if $\sigma(\delta_0)=0$.
\end{definition}

\begin{remark}
We could replace all the $\cP_b$ in \cref{def:DCRM} with $\cP$ so that the definition covers probabilities that may not have bounded support. As we subsequently assume boundedness, however, the regularities of such extended definition will be investigated in future work.
\end{remark}

\begin{remark}\label{rmk:DCRM}
In this remark, we suggest that after adopting \cref{def:DCRM}, we retain access to some of the important results previously established for law-invariant convex risk measures defined on the space of random variables. Below we consider $\sigma$ as in \cref{def:DCRM}, and set $\rho(Z):=\sigma(\cL(Z))$.
\begin{itemize}
\item 
Here we provide some discussion related to the Fatou's property of law-invariant convex risk measure (cf. \cite[Theorem 7]{Frittelli2005Law}). Suppose $(\mu^{n})_{n\in\bN}\subset\cP_b(\bR)$ converges weakly to $\mu^0\in\cP_b(\bR)$, and there is a bounded $B\subset\bR$ such that $\supp \mu^n\subseteq B$ for $n\in\bN$. Then,
$$\liminf_{n\to\infty}\sigma(\mu^n)\ge\sigma(\mu^0).$$ 
To see this, we first observe that $\rho$ is a law-invariant convex risk measure. Next, in view of the Skorohod representation (cf. \cite[Theorem 8.5.4]{Bogachev2007book}), there is a sequence $(Z^n)_{n\in\bN}\subset L^\infty(\Omega,\sH,\bP)$ such that $\cL(Z^n)=\mu^n$ for $n\in\bN$ and $(Z^n)_{n\in\bN}$ converges $\bP$-a.s. Note that $\bP(Z_n\in B)=1$. Let $Z^0$ be the pointwise limit of $(Z^n)_{n\in\bN}$. As almost sure convergence implies convergence in distribution, we have $\cL(Z^0)=\mu^0$. Invoking the Fatou's property of $\rho$, we prove the claim.
\item By applying the Kusuoka representation on $\rho$ (cf. \cite[Theorem 7]{Frittelli2005Law}), there is a convex $\beta:\cP([0,1])\to[0,\infty]$ such that
\begin{align}\label{eq:DCRMKusuoka}
\sigma(\mu)=\sup_{\eta\in\cP([0,1])} \left\{ \int_{[0,1]} \overline\avar_{\kappa}(\mu)\,\eta(\dif\kappa) - \beta_t(\eta) \right\}
\end{align}
for any $\mu\in\cP_b(\bR)$, where
\begin{align}\label{eq:DefAVaR}
\overline\avar_{\kappa}(\mu) := \inf_{q\in\bR}\left\{ q + \kappa^{-1} \int_{\bR} (z-q)_+\,\mu(\dif z) \right\}, \quad \kappa\in(0,1]
\end{align}
and
\begin{align}\label{eq:DefAVAR0}
\overline\avar_{0}(\mu) := \inf\set{r\in\bR:\mu((r,\infty))=0}.
\end{align}
Note \eqref{eq:DefAVaR} follows from the extremal representation of average value-at-risk in \cite{Rockafellar2000Optimization}, and \eqref{eq:DefAVAR0} is the essential supremum of a real-valued random variable with distribution $\mu$.
\item Let $\mu,\mu'\in\cP_b(\bR)$. We say $\mu$ is dominated by $\mu'$ in convex order if for any convex $g:\bR\to\bR$, we have $\int_{\bR} g(r)\,\mu(\dif r)\le\int_{\bR} g(r)\,\mu'(\dif r)$. It follows immediately from \eqref{eq:DefAVaR} and \eqref{eq:DefAVAR0} that for $\kappa\in[0,1]$, $\overline\avar_{\kappa}$ is consistent with convex order. Consequently, $\sigma$ is also consistent with convex order.
\end{itemize}
\end{remark}

\subsection{DRMs at the level of  distributions}\label{subsec:DDRM}

In this section, we define DRMs at the level of distribution as a nested composition of law-invariant convex risk measures specified in \cref{def:DCRM}. To achieve this, we adhere to \cref{def:DCRM}, and consider $\sigma_0$ and $(\sigma_{t,x})_{(t,x)\in\bN\times\bX}$ as risk measures that are employed at various times and states.  We make the following standing assumption throughout the paper.
\begin{assumption}\label{assump:mbl}
For any $t\in\bN$ and compact $K\subset\bR$, $(x,\mu)\in \bX\times\cP(K)\mapsto\sigma_{t,x}(\mu)\in(-\infty,\infty]$ is $\cB(\bX)\otimes\cB(\cP(K))$-$\cB(\overline{\bR})$ measurable and for each $t\in\bN$ there is a $b_t\ge 0$ such that $|\sigma_{t,x}(\delta_0)|\le b_t$ for all $x\in\bX$. 
\end{assumption}
We refer to \cref{subsec:ExampleSigma} for two examples of $\sigma_{t,x}$.

Let $\gamma\in[0,1]$  be the discount factor. Recall the filtrations defined in the beginning of \cref{sec:Setup}. In particular, the $\sigma$-algebra, $\sU_t$, used below stands for the information available to the agent at time $t$. For $\fZ=(Z_t)_{t\in\bN_0}\subset L^\infty(\Omega,\sH,\bP)$, we first define the finite horizon DRM 
\begin{align}\label{eq:DefsigmatT}
\varsigma_{t,x,T}^{\fX,\gamma}(\fZ) :=
\begin{cases}
\sigma_{t,x}\left(P^{W^T_t|\sU_t}\right), & 0<t<T,\\
\sigma_{T,x}\left(P^{Z_T|\sU_T}\right), & t=T,
\end{cases}
\end{align}
where for $0<t<T$
$$
W_t^T=Z_t + \gamma\;\varsigma^{\fX,\gamma}_{t+1,X_{t+1},T}(\fZ).
$$ 
For $t=0$ and $T\in\bN$, we define $$\varsigma^{\fX,\gamma}_{0,T}(\fZ):=\sigma_{0}\left(
\cL\left(Z_0+\gamma\;\varsigma^\fX_{1,X_{1},T}(\fZ) \right)
\right).$$
Notably, the $x$ entry in $\sigma_{t,x}$ provides us with the flexibility of adjusting risk aversion based on the value of the state process $\fX$ at that time. This allows for situational adaptation and varying degrees of risk aversion depending on the circumstances.  In \cref{lem:finiteTsigma} below, we validate the definition in \eqref{eq:DefsigmatT} for $0<t<T$ . The validity of $\varsigma^{\fX,\gamma}_{0,T}(\fZ)$ follows automatically once \cref{lem:finiteTsigma} is proved. Below, we say $(\sigma_{t,x})_{(t,x)\in\bN\times\bX}$ is normalized if $\sigma_{t,x}$ is normalized for each $(t,x)\in\bN\times\bX$. Obviously, $(\sigma_{t,x})_{(t,x)\in\bN\times\bX}$ is normalized if and only if $b_t=0$ (where $b_t$ is defined in \Cref{assump:mbl}) for $t\in\bN$.
\begin{lemma}\label{lem:finiteTsigma}
Let $\fZ=(Z_t)_{t\in\bN_0}\subset L^\infty(\Omega,\sH,\bP)$ and $t,T\in\bN$ with $0<t<T$. We have $(x,\omega)\mapsto \varsigma_{t,x,T}^{\fX,\gamma}(\fZ) (\omega)$ is $\cB(\bX)\otimes\sU_t$-$\cB(\bR)$ measurable and $\|W_t^T\|_\infty\le \sum_{r=t}^{T} \gamma^{r-t} (b_t+\|Z_r\|_\infty)$. Consequently, if $(\sigma_{t,x})_{(t,x)\in\bN\times\bX}$ is normalized, then $\|W_t^T\|_\infty\le \sum_{r=t}^{T} \gamma^{r-t}\|Z_r\|_\infty$.
\end{lemma}
\begin{proof}
We start by analyzing the measurability of $(x,\omega)\mapsto \sigma_{T,x}\left(P^{Z_T|\sU_T}(\omega,\cdot)\right)$. Note that $Z_T\in[-\|Z_T\|_\infty,\|Z_{T}\|_\infty]$, $\bP$-a.s. Then, by \eqref{eq:RegCondProbInt},
\begin{align}\label{eq:CondDistnZUT}
P^{Z_T|\sU_{T}}\in\cP([-\|Z_T\|_\infty,\|Z_T\|_\infty]),\quad\bP-\text{a.s.}
\end{align}
Let us momentarily consider a version of $P^{Z_T|\sU_T}$, denoted by $\overline P^{Z_T|\sU_T}$, such that \eqref{eq:CondDistnZUT} is true for all $\omega$. For this version, in view of the discussion in \cref{subsec:RegCondDist},  we have $\omega\mapsto \overline P^{Z_T|\sU_T}(\omega,\cdot)$ is $\sU_T$-$\cB(\cP([-\|Z_T\|_\infty,\|Z_T\|_\infty]))$ measurable and thus $(x,\omega)\mapsto \sigma_{T,x}\big(\overline P^{Z_T|\sU_T}(\omega,\cdot)\big)$ is $\cB(\bX)\otimes\sU_T$-$\cB(\bR)$ measurable. Moreover, due to the uniqueness of regular conditional distribution discussed in \cref{subsec:RegCondDist}, it is true for $\bP$-almost every $\omega$ that for any $x\in\bX$, $\sigma_{T,x}\left( P^{Z_T|\sU_T}(\omega,\cdot)\right)=\sigma_{T,x}\big(\overline P^{Z_T|\sU_T}(\omega,\cdot)\big)$. Since $\sF_1$ contains all the $\bP$-negligible sets and $\sF_t\subseteq \sU_T$, the measurability is de facto version independent. 
The above, together with \eqref{eq:DefsigmatT}, \cref{def:DCRM} (i) and (ii), implies that
\begin{align*}
\left|\varsigma^{\fX,\gamma}_{T,X_T,T}(\fZ)\right|\le b_T+\|Z_T\|_\infty,\quad\bP\text{-a.s.}
\end{align*}
Consequently,
\begin{align}\label{eq:Wbound}
\left|W_{T-1}^{T}\right| = \left|Z_{T-1}+\gamma\varsigma^{\fX,\gamma}_{T,X_{T},T}(\fZ)\right|\le \|Z_{T-1}\|_\infty +  \gamma(b_{T}+\|Z_{T}\|_\infty),\quad\bP\text{-a.s.}
\end{align}
Inducting backward with a similar reasoning, we conclude the proof.
\end{proof}

The finite horizon DRM defined in \eqref{eq:DefsigmatT} satisfies the time consistency property below.
\begin{proposition}\label{prop:TimeConsistency}
Let $0\le s\le t < T$. Consider $\fZ=(Z_n)_{n\in\bN_0}, \fZ'=(Z'_n)_{n\in\bN_0}\subset L^\infty(\Omega,\sH,\bP)$. If $\bP(Z_r=Z_r')=1$ for $r=s,\dots,t$ and $\bP\big(\varsigma_{t+1,x,T}^{\fX,\gamma}(\fZ) \le \varsigma_{t+1,x,T}^{\fX,\gamma}(\fZ'),\,x\in\bX\big)=1$,  then $\bP\big(\varsigma_{s,x,T}^{\fX,\gamma}(\fZ) \le \varsigma_{s,x,T}^{\fX,\gamma}(\fZ'),\,x\in\bX\big)=1$. 
\end{proposition}
\begin{proof}
The proof is conceptually straightforward, with the exception of some technicalities needed to address the regular conditional distributions. It is sufficient to prove the case of $s=t>0$ as the case of $s<t$ and/or $s=0$ can be proved by induction. We fix $T$ for the remainder of the proof. In view of \eqref{eq:DefsigmatT}, let $W_t=Z_t + \gamma\;\varsigma^{\fX,\gamma}_{t+1,X_{t+1},T}(\fZ)$ and $W'_t=Z_t + \gamma\;\varsigma^{\fX,\gamma}_{t+1,X_{t+1},T}(\fZ')$. Let $\overline P^{(W_t,W_t')|\sU_t}$ be the regular conditional distribution of $(W_t,W_t')$ given $\sU_t$. Since $W_t\le W'_t$, $\bP$-a.s., by \cref{eq:DefsigmatT}, \cref{def:RegCondDist} (iii), $\int_{\bR^2}\1_{(0,\infty)}(w_1-w_2) \overline P^{(W_t,W_t')|\sU_t}(\dif w_1 \dif w_2)=0$, $\bP$-a.s. It follows that
\begin{align}\label{eq:PWPW'}
\bP\left(\overline P^{W_t|\sU_t}((-\infty,r]) \ge \overline P^{W_t'|\sU_t}((-\infty,r],\,r\in\bR\right) = 1,
\end{align}
where $\overline P^{W_t|\sU_t}:=\overline P^{W_t,W'_t|\sU_t}(\cdot\times\bR)$ and $\overline P^{W'_t|\sU_t}:=\wt P^{W_t,W'_t|\sU_t}(\bR\times\cdot)$ are the first and second marginal distribution of $P^{W_t,W_t'|\sU_t}$, respectively. Note additionally that $\overline P^{W_t|\sU_t}$ and $\overline P^{W'_t|\sU_t}$ are respectively the regular conditional distributions of $W_t$ and $W_t'$ given $\sU_t$. Since regular conditional distributions are unique up to a $\bP$-negligible set, by \eqref{eq:PWPW'} and \cref{def:DCRM} (ii), 
we have
\begin{align*}
\bP\left(\varsigma_{t,x,T}^{\fX,\gamma}(\fZ) \le \varsigma_{t,x,T}^{\fX,\gamma}(\fZ'),\,x\in\bX \right) = \bP\left(\sigma_{t,x}\left(P^{W_t|\sU_t}\right) \le \sigma_{t,x}\left(P^{W'_t|\sU_t}\right),\, x\in\bX \right)=1.
\end{align*}
\end{proof}

In what follows, we define the infinite horizon DRM. Consider $\fZ=(Z_t)_{t\in\bN}\subset L^\infty(\Omega,\sH,\bP)$ such that $\|\fZ\|_\infty:=\sup_{t\in\bN_0}\|Z_t\|_\infty<\infty$. Additionally, we assume $\gamma\in[0,1)$ and suppose that $(\sigma_{t,x})_{(t,x)\in\bN\times\bX}$ is normalized. We note that $(\sigma_{t,x})_{(t,x)\in\bN\times\bX}$ being normalized is not necessarily required for the construction of infinite horizon DRM, but normalization greatly eases the analysis. We define
\begin{align}\label{eq:DefsigmatInfty}
\varsigma^{\fX,\gamma}_{0,\infty}(\fZ):=\lim_{T\to\infty}\varsigma^{\fX,\gamma}_{0,T}(\fZ) \quad \text{and} \quad \varsigma_{t,x,\infty}^{\fX,\gamma}(\fZ):=\lim_{T\to\infty}\varsigma_{t,x,T}^{\fX,\gamma}(\fZ),\quad(t,x)\in\bN\times\bX.
\end{align}  
The validity of the definitions in\eqref{eq:DefsigmatInfty} are justified in \cref{lem:sigmatTConv} below. For the remainder of this paper, we default to the $\cB(\bX)\otimes\sU_t$-$\cB(\bR)$ measurable version whenever we engage with $\varsigma^{\fX,\gamma}_{t,x,\infty}(\fZ)$.
\begin{lemma}\label{lem:sigmatTConv}
Suppose that $\sup_{t\in\bN_0}\|Z_t\|_\infty<\infty$, $\gamma\in[0,1)$ and $(\sigma_{t,x})_{(t,x)\in\bN\times\bX}$ is normalized.  For any $t\in\bN$, we have $\bP\big(\big(\varsigma^{\fX,\gamma}_{t,x,T}(\fZ)\big)_{T\in\bN}\text{ converges for any } x\in\bX\big)=1$ and $\bP\big(\big|\varsigma^{\fX,\gamma}_{t,x,\infty}(\fZ)\big| \le (1-\gamma)^{-1}\|\fZ\|_\infty,\, x\in\bX\big)=1$.
Moreover, there is a $\overline\varsigma^{\fX,\gamma}_{t,x,\infty}(\fZ)$ such that $\bP\big(\varsigma^{\fX,\gamma}_{t,x,\infty}(\fZ)=\overline\varsigma^{\fX,\gamma}_{t,x,\infty}(\fZ),\,x\in\bX\big)=1$ and $(x,\omega)\mapsto\overline\varsigma^{\fX,\gamma}_{t,x,\infty}(\fZ)(\omega)$ is $\cB(\bX)\otimes\sU_t$-$\cB(\bR)$ measurable. Finally, $\varsigma^{\fX,\gamma}_{0,T}(\fZ)$ converges as $T\to\infty$ and $\varsigma^{\fX,\gamma}_{0,\infty}(\fZ)$ is finite.
\end{lemma}
\begin{proof}
By \cref{lem:finiteTsigma} and that $(\sigma_{t,x})_{t\in\bN}$ is normalized, $\|W_T^{T+r}\|_\infty\le (1-\gamma)^{-1}\|\fZ\|_\infty$ for any $(T,r)\in\bN^2$. By \cref{def:DCRM} (i) (ii), we have $\|\varsigma^{\fX,\gamma}_{T,X_{T},T+r}(\fZ)\|_\infty\le (1-\gamma)^{-1}\|\fZ\|_\infty$.  
Consequently,
\begin{align*}
 Z_{T-1} - \frac{\gamma}{1-\gamma}\|\fZ\|_\infty \le Z_{T-1}+\gamma\;\varsigma^{\fX,\gamma}_{T,X_{T},T+r}(\fZ) \le Z_{T-1} + \frac{\gamma}{1-\gamma}\,\|\fZ\|_\infty, \quad\bP\text{-a.s.}
\end{align*}
By \eqref{eq:DefsigmatT} and \cref{def:DCRM} (i) (ii), with a similar reasoning as in the proof of \cref{prop:TimeConsistency}, we obtain
\begin{align*}
\bP\left(\varsigma^{\fX,\gamma}_{T-1,x,T}(\fZ) - \frac{\gamma}{1-\gamma}\|\fZ\|_\infty \le \varsigma^{\fX,\gamma}_{T-1,x,T+r}(\fZ) \le \varsigma^{\fX,\gamma}_{T-1,x,T}(\fZ) + \frac{\gamma}{1-\gamma}\|\fZ\|_\infty,\,x\in\bX\right) = 1.
\end{align*}
We next proceed by backward induction. Suppose for some $t<T$, it is true for $r\in\bN$ that 
\begin{align*}
\bP\left(\varsigma^{\fX,\gamma}_{t,x,T}(\fZ)-\frac{\gamma^{T-t}}{1-\gamma}\|\fZ\|_\infty \le \varsigma^{\fX,\gamma}_{t,x,T+r}(\fZ) \le \varsigma^{\fX,\gamma}_{t,x,T}(\fZ) + \frac{\gamma^{T-t}}{1-\gamma}\|\fZ\|_\infty,\,x\in\bX\right) = 1\,.
\end{align*}
Note that 
\begin{align*}
W_{t-1}^{T+r} &= Z_{t-1} + \gamma \,\varsigma^{\fX,\gamma}_{t,x,T+r}(\fZ) 
\\
&\le Z_{t-1} + \gamma \,\varsigma^{\fX,\gamma}_{t,x,T}(\fZ) + 
\frac{\gamma^{T-(t-1)}}{1-\gamma}\,
\|\fZ\|_\infty = W_{t-1}^T + 
\frac{\gamma^{T-(t-1)}}{1-\gamma}\,\|\fZ\|_\infty,\quad\bP\text{-a.s.}
\end{align*}
Similarly, we have $W_{t-1}^{T+r}\ge W_{t-1}^T - \gamma^{T-(t-1)}(1-\gamma)^{-1}\|\fZ\|_\infty,\,\bP$-a.s. Handling the conditional distributions of $W_{t-1}^{T+r}$ and $W_{t-1}^{T}$ given $\sU_{t-1}$ with the same technique from the proof of \cref{prop:TimeConsistency}, by \cref{def:DCRM} (i) (ii), we obatin
\begin{align*}
\bP\left(\varsigma^{\fX,\gamma}_{t-1,x,T}(\fZ)-\frac{\gamma^{T-(t-1)}}{1-\gamma}\|\fZ\|_\infty \le \varsigma^{\fX,\gamma}_{t-1,x,T+r}(\fZ) \le \varsigma^{\fX,\gamma}_{t-1,x,T}(\fZ) + \frac{\gamma^{T-(t-1)}}{1-\gamma}\|\fZ\|_\infty,\,x\in\bX\right)=1.
\end{align*}
Therefore, for any $(t,T,r)\in\bN^3$ with $t\le T$, we have
\begin{align*}
\bP\left(\left|\varsigma^{\fX,\gamma}_{t,x,T+r}(\fZ)-\varsigma^{\fX,\gamma}_{t,x,T}(\fZ)\right| \le \frac{\gamma^{T-t}}{1-\gamma}\,\|\fZ\|_\infty,\,x\in\bX \right)=1,
\end{align*} 
and thus
\begin{align}\label{eq:sigmaConv}
&\bP\left(\big(\varsigma^{\fX,\gamma}_{t,x,T}(\fZ)\big)_{T\in\bN}\text{ converges for any } x\in\bX\right)\nonumber \\
&\quad\ge \bP\left(\bigcap_{T\ge t}\bigcap_{r\in\bN}\left\{\left|\varsigma^{\fX,\gamma}_{t,x,T+r}(\fZ)-\varsigma^{\fX,\gamma}_{t,x,T}(\fZ)\right| \le \frac{\gamma^{T-t}}{1-\gamma}\|\fZ\|_\infty,\,x\in\bX\right\} \right)=1.
\end{align}
In view of \cref{lem:finiteTsigma}, with similar reasoning as before, we have 
\begin{align*}
\bP\left(\bigcap_{T\ge t}\left\{\left|\varsigma^{\fX,\gamma}_{t,x,T}(\fZ)\right| \le \frac{1}{1-\gamma}\,\|\fZ\|_\infty,\,x\in\bX\right\} \right)=1,
\end{align*}
which together with \eqref{eq:sigmaConv} implies the boundedness. 
Regarding the measurability, by \eqref{eq:sigmaConv} and the setting that $\sF_1$ contains all the $\bP$-neglegible sets, we have 
\begin{align*}
\cN:=\left\{\omega\in\bX\times\Omega: \big(\varsigma^{\fX,\gamma}_{t,x,T}(\fZ)\big)_{T\in\bN}\text{ converges for any } x\in\bX\right\}^c \in \sU_t.
\end{align*}
We let $\overline\varsigma^{\fX,\gamma}_{t,x,T}(\fZ) := \1_{\cN^c}\;\varsigma^{\fX,\gamma}_{t,x,T}(\fZ)$ and define $\overline\varsigma^{\fX,\gamma}_{t,x,\infty}=\lim_{T\to\infty}\overline\varsigma^{\fX,\gamma}_{t,x,T}(\fZ)$. Then, $\bP\big(\varsigma^{\fX,\gamma}_{t,x,\infty}(\fZ)=\overline\varsigma^{\fX,\gamma}_{t,x,\infty}(\fZ),\,x\in\bX\big)=1$. Note the convergence in the definition of $\overline\varsigma^{\fX,\gamma}_{t,x,\infty}$ is pointwise in $(x,\omega)\in\bX\times\Omega$. This together with \cref{lem:finiteTsigma} and the fact that pointwise convergence preserves measurability (cf. \cite[Section 4.6, Lemma 4.29]{Aliprantis2006book}) implies that $\overline\varsigma^{\fX,\gamma}_{t,x,\infty}$ is $\cB(\bX)\otimes\sU_t$-$\cB(\bR)$ measurable. Finally, the statement for $\varsigma^{\fX,\gamma}_{0,\infty}(\fZ)$ follows analogously.
\end{proof}

To conclude this section, we demonstrate that $(\varsigma_{t,x,\infty}^{\fX,\gamma})_{t\in\bN_0}$ adheres to a similar recursive relationship and exhibits the same time consistency as its finite horizon counterpart.
\begin{proposition}%\label{prop:InfHorizonProp}
Consider $\fZ=(Z_n)_{n\in\bN_0}, \fZ'=(Z'_n)_{n\in\bN_0}\subset L^\infty(\Omega,\sH,\bP)$ with $\|\fZ\|_\infty,\|\fZ'\|_\infty<\infty$. Then, 
\begin{align}\label{eq:sigmatInftyRecursive}
\bP\left(\varsigma_{t,x,\infty}^{\fX,\gamma}(\fZ) = \sigma_{t,x}(P^{W^\infty_t|\sU_t}),\,x\in\bX\right)=1,
\end{align}
where $W^\infty_t = Z_t + \gamma \varsigma_{t+1,X_{t+1},\infty}^{\fX,\gamma}(\fZ)$. Moreover, let $s\in\bN_0$. If $\bP(Z_r=Z_r')=1$ for $r=s,\dots,t$ and $\bP\big(\varsigma_{t+1,x,\infty}^{\fX,\gamma}(\fZ) \le \varsigma_{t+1,x,\infty}^{\fX,\gamma}(\fZ'),\,x\in\bX\big)=1$,  then $\bP\big(\varsigma_{s,x,\infty}^{\fX,\gamma}(\fZ) \le \varsigma_{s,x,\infty}^{\fX,\gamma}(\fZ'),\,x\in\bX\big)=1$. 
\end{proposition}
\begin{proof}
By \eqref{eq:sigmaConv} and the observation that 
\begin{align*}
\bigcap_{T\ge t}\bigcap_{r\in\bN} \left\{\left|\varsigma^{\fX,\gamma}_{t,x,T+r}(\fZ)-\varsigma^{\fX,\gamma}_{t,x,T}(\fZ)\right| \le \frac{\gamma^{T-t}}{1-\gamma}\|\fZ\|_\infty,\,x\in\bX\right\}\\
\subseteq \bigcap_{T\ge t} \left\{\left|\varsigma^{\fX,\gamma}_{t,x,\infty}(\fZ)-\varsigma^{\fX,\gamma}_{t,x,T}(\fZ)\right| \le \frac{\gamma^{T-t}}{1-\gamma}\|\fZ\|_\infty,\,x\in\bX\right\},
\end{align*}
we have 
\begin{gather}
\bP\left(\left|\varsigma^{\fX,\gamma}_{t,x,\infty}(\fZ)-\varsigma^{\fX,\gamma}_{t,x,T}(\fZ)\right| \le \frac{\gamma^{T-t}}{1-\gamma}\|\fZ\|_\infty,\,x\in\bX\right) = 1,\label{eq:sigmatTConvRate} \\
\left| W^\infty_t - \left( Z_t + \gamma \varsigma_{t+1,X_{t+1},T}^{\fX,\gamma}(\fZ) \right) \right| \le \frac{\gamma^{T-(t+1)}}{1-\gamma}\|\fZ\|_\infty,\quad \bP\text{-a.s.} \label{eq:WtInftyRange}
\end{gather}
By combining \cref{def:DCRM} (i) (ii), \eqref{eq:DefsigmatT} and \eqref{eq:WtInftyRange}, with a similar reasoning as in the proof of \cref{prop:TimeConsistency}, we obtain
\begin{align*}
\bP\left( \left| \sigma_{t,x}(P^{W^\infty_t|\sU_t}) - \varsigma_{t,x,T}^{\fX,\gamma}(\fZ) \right| \le \frac{\gamma^{T-(t+1)}}{1-\gamma}\|\fZ\|_\infty,\, x\in\bX \right) = 1.
\end{align*}
This together with \cref{eq:sigmatTConvRate} implies that for any $t<T$,
\begin{align*}
\bP\left( \left| \sigma_{t,x}(P^{W^\infty_t|\sU_t}) - \varsigma_{t,x,\infty}^{\fX,\gamma}(\fZ) \right| \le \frac{(1+\gamma)\gamma^{T-(t+1)}}{1-\gamma}\|\fZ\|_\infty,\, x\in\bX \right) = 1,
\end{align*}
which proves \eqref{eq:sigmatInftyRecursive}. The proof for the remaining statement on time consistency can be directly derived from the arguments used in \cref{prop:TimeConsistency}, and is therefore omitted here.
\end{proof}

\subsection{Problem formulation}\label{subsec:Problem}
Here, we provide some standing assumptions and remarks on the key problem that we address: how can the DRM constructed in \cref{subsec:DDRM} be optimised over actions?

We let $\Psi$ be a subset of $(\fX,\fA)$ satisfying \eqref{eq:GMarkov} and $\bP(A_t\in\cA_t(X_t),\,t\in\bN)=1$ (note that by \cite[Section 18.1, Theorem 18.6]{Aliprantis2006book}, $\set{(x,a)\in\bX\times\bA:a\in\cA_t(x)}\in\cB(\bX)\otimes\cB(\bA)$). Throughout the rest of the paper, we make the following standing assumption.\footnote{For a nontrivial example, one can set $\bX=\bA=\bR$ and construct $(\fX,\fA)$ from any given $P$ and $\fp$ on a complete probability space that supports a countable family of mutually independent $U([0,1])$ random variables. The construction can be done by quantile transform (cf. \cite[Theorem 1]{OBrien1975Comparison}). For examples of $\bX$ and $\bA$ that are complete separable metric spaces, we refer to \cite{Blackwell1983Extension}.}
\begin{assumption}\label{asmp:PlaceHolder}
The family $\set{(\Omega,\sH,\bP),\bF,\bG,\Psi}$ satisfies the conditions below:
\begin{itemize}
\item[(i)] There exist a probability measure on $\cB(\bX)$, denoted by $\mu$, and a transition kernel on $\cB(\bX)$, denoted by $P$, such that for any $(\fX,\fA)\in\Psi$, we have $X_1\sim\mu$ and \eqref{eq:transkernel} holds true for any $t\in\bN$. Above, $P$ should satisfy
\begin{itemize}
\item $P(t,x,a,\cdot)$ is a probability measure on $\cB(\bX)$ for any $(t,x,a)\in\bN\times\bX\times\bA$;
\item $(x,a)\mapsto P(t,x,a,B)$ is $\cB(\bX)\otimes\cB(\bA)$-measurable for any $t\in\bN$ and $B\in\cB(\bX)$.
\end{itemize}
\item[(ii)] For any Markovian policy $\fp=(\pi_t)_{t\in\bN}\in\Pi$, there is $(\fX,\fA)\in\Psi$ such that \eqref{eq:MarkovControlkernel} holds true for all $t\in\bN$.
\end{itemize}
\end{assumption}
When $\fA$ is associated with some Markovian policy $\fp\in\Pi$ via \eqref{eq:MarkovControlkernel}, we  write $(\fX^{\fp}, \fA^{\fp}) = \set{(X^{\fp}_t, A^{\fp}_t)}_{t\in\bN}$ to emphasize the association with $\fp$.

At each $t\in\bN$, we are given a cost function $C_t:(\bX\times\bA\times\bX,\cB(\bX)\otimes\cB(\bA)\otimes\cB(\bX))\to(\bR,\cB(\bR))$ and we force $C_0\equiv 0$. Let us define $\fC:=(C_t(X_t,A_t,X_{t+1}))_{n\in\bN_0}$.

Below is our main goal. 
\begin{align}\label{eq:main}\tag{P}
\text{For $T\in\bN\cup\set{\infty}$,\,\, find \,\, $\inf_{(\fX,\fA)\in\Psi}\varsigma^{\fX,\gamma}_{0,T}(\fC)$, \,\, and the optimal policy if exists.}
\end{align}
For $T\in\bN$, it is a finite horizon problem. For $T=\infty$, it is an infinite horizon problem and, as seen in \cref{subsec:DDRM}, the definition of $\varsigma^{\fX,\gamma}_{0,\infty}$ is based on the assumptions that $\gamma\in[0,1)$ and $(\sigma_{t,x})_{(t,x)\in\bN\times\bX}$ is normalized. 

\begin{remark}
At times, the choice of discount factor in an infinite horizon can be controversial. One approach to address this is to consider the existence of a Markovian policy that is optimal for all $\gamma\in(\delta, 1)$ for some $\delta\in(0,1)$, which is known as Blackwell optimality (cf. \cite{Blackwell1962Discrete}). In a risk-neutral setting with discrete state and action spaces, Blackwell optimality does exist. The proof hinges on that the value function uniquely solves an affine equation; see also \cite[Seciont 10.1]{Puterman2005Markov}, \cite[Section 3.1]{Grand-Clement2023Reducing}. However, in a risk-averse setting, where the affine property is often absent, Blackwell optimality can be elusive. Another potentially promising approach to address uncertainty in the discount factor $\gamma$ involves considering the limit of $\varsigma^{\fX,\gamma}_{0,\infty}(\fC(\fX,\fA))$ as $\gamma \to 1$. In the risk-neutral setting with discrete state and action spaces, it has been shown that this limit while scaled properly coincides with average optimality (cf. \cite[Corollary 8.2.5]{Puterman2005Markov}, \cite{Dewanto2022Examining}). However, the equivalence of this limit in current context remains unclear.
\end{remark}

In what follows, we fix $\gamma\in[0,1]$ for finite horizon problems and $\gamma\in[0,1)$ for infinite horizon problems. The variable $\gamma$ will subsequently be omitted from the notation $\varsigma^{\fX,\gamma}_{t,T}$.

In addition to  \cref{asmp:PlaceHolder}, we invoke further technical assumptions for the derivation of the DPP to be rigorous.
\begin{assumption}\label{asmp:Main} The following is true for any $t\in\bN$:
\begin{itemize}
\item[(i)] $(x,a)\mapsto P(t,x,a,\cdot)$ is weakly continuous, that is, for any $(x^n)_{n\in\bN}\subseteq\bX$ and $(a^n)_{n\in\bN}\subseteq\bA$ such that $\lim_{n\to\infty}x^n=x^0$ and $\lim_{n\to\infty}a^n=a^0$, we have
\begin{align*}
\lim_{n\to\infty}\int_{\bR}f(y)\,P(t,x^n,a^n,\dif y) = \int_{\bR}f(y)\,P(t,x,a,\dif y),\quad f\in C_b(\bX). 
\end{align*}

\item[(ii)] 
$\bigcup_{x\in\bX}\cA_t(x)$ is compact for $x\in\bX$ and $\cA_t$ is upper hemi-continuous, that is, at any $x\in\bX$ for every open $U_\bA\supseteq\cA_t(x)$ there is a open $U_\bX\ni x$ such that $z\in U_\bX$ implies $\cA_t(z)\subseteq U_\bA$.

\item[(iii)] For any compact $K\subset\bR$, $(\mu^{n})_{n\in\bN}\subset\cP_b(K)$ converging weakly to $\mu^0\in\cP(K)$ and $(x^n)_{n\in\bN}\subseteq\bX$ that converges to $x^0\in\bX$, we have $\liminf_{n\to\infty}\sigma_{t,x^n}(\mu^n)\ge \sigma_{t,x^0}(\mu^0)$.

\item[(iv)] The cost function $C_t$ is lower semi-continuous and $\|C_t\|_\infty \le b$ for some $b>0$.
\end{itemize}
\end{assumption}
Note Assumption \ref{asmp:Main} (ii) implies that $\bigcup_{x\in\bX}\varpi_t(x)$ is compact (cf. \cite[Section 15.6, Theorem 15.22]{Aliprantis2006book}) and $\varpi_t$ is upper hemi-continuous (cf. \cite[Section 17.2, Theorem 17.13]{Aliprantis2006book}). %By closed graph theorem (cf. \cite[Section 17.2, Theorem 17.11]{Aliprantis2006book}), $\varpi_t$ has a closed graph.

\begin{remark}
The preceding assumptions do not lead to any contradiction. Indeed, several earlier assumptions can be derived from \cref{asmp:Main}. Specifically, by \cref{asmp:Main} (i), $(x,a)\mapsto P(t,x,a,\,\cdot\,)$ is $\cB(\bX)\otimes\cB(\bA)$-$\cB(\Xi)$ measurable. This along with \cref{lem:sigmaAlgBE} ensures that any probability measure $P$ that satisfies \cref{asmp:Main} (i) also meets the conditions on $P$ as stated in \cref{asmp:PlaceHolder} (i). In our current framework where the input space $\bX$ is equipped with $\cB(\bX)$, the upper hemi-continuity in \cref{asmp:Main} (ii) implies weak measurability, as delineated in \cite[Section 17.2, Lemma 17.4 and Section 18.1, Lemma 18.2]{Aliprantis2006book}. Finally, \cref{asmp:Main} (iii) implies the measurability in \cref{assump:mbl} as $\set{(x,\mu)\in\bX\times\cP_K(\mu):\sigma_{t,x}(\mu)\ge c}$ is closed for any $c\in\overline\bR$.
\end{remark}

\section{Auxiliaries}\label{sec:Aux}

%\seb{If the whole of section 3 is under this assumption, let's state that up front.}\ziteng{do you mean changing the introduction? I added a small revision in brown}
Temporarily, let us restrict our attention to Markovian policies $\fp$ and investigate the finite horizon $\varsigma_{t,x,T}^{\fX^\fp}({\fC^\fp})$ 
%\seb{and similarly... could we just use $\fC^\fp:=(C_t(X_t^\fp,A_t^\fp,X_{t+1}^\fp))_{n\in\bN_0}$ to reduce notation a bit?}\zit{Done.} 
introduced in \eqref{eq:DefsigmatT}, where $\fC^\fp:=\big(C_t(X_t^\fp,A_t^\fp,X_{t+1}^\fp)\big)_{t\in\bN_0}$. The Markovian nature of the controlled process provides a way to express $\varsigma^{\fX^\fp}_{t,x,T}({\fC^\fp})$ as composition of operators, which in turn facilitates the derivation of a DPP. The various operators involved are introduced below. 

For $t=0$ we define a functional $H_0:\ell^\infty(\bX,\cB(\bX))\to\bR$ as
\begin{align}\label{eq:DefH0}
H_0 v &:= \sigma_0\left(\wt P^{v}_{0}\right), 
\end{align}
where $\wt P^{v}_{0}$ is the distribution of $\gamma\,v(Y)$ and $Y\sim\xi_1$. 
For $t\in\bN$ and $\lambda\in\Lambda$, we define an operator $G^{\lambda}_t:\ell^\infty(\bX,\cB(\bX))\to\ell^\infty(\bX,\cB(\bX))$ as 
\begin{align}\label{eq:DefG}
&G^{\lambda}_t v(x) := \sigma_{t,x}\left(\wt P^v_{t,x,\lambda}\right),
\end{align}
where $\wt P^{v}_{t,x,\lambda}$ is the distribution of $C_t(x,A,Y)+\gamma\, v(Y)$, and $A\sim \lambda$ and $Y\sim \int_{\bA}P(t,x,a,\cdot)\lambda(\dif a)$. More precisely, 
\begin{align}\label{eq:DefPtilde}
\wt P^{v}_{t,x,\lambda}(B) = \int_{\bA}\int_{\bX}\1_{B}\big(C_t(x,a,y)+\gamma v(y)\big)
\;P(t,x,a,\dif y)\;\lambda(\dif a), 
\quad B\in\cB(\bR).
\end{align}
Although $H_0$ and $G^\lambda_t$ depend on $\gamma$ through \eqref{eq:DefPtilde}, as $\gamma$ is fixed we omit $\gamma$ from the notation of $H_0$ and $G^\lambda_t$. Furthermore, for $t\in\bN$ and $\fp\in\Pi$, we define an operator $H^\fp_t:\ell^\infty(\bX,\cB(\bX))\to\ell^\infty(\bX,\cB(\bX))$ as 
\begin{align}\label{eq:DefH}
H^{\fp}_t v(x) := G^{\pi_t(x)}_t v(x).
\end{align}

We first delve into the fundamental properties of the operators introduced previously. \cref{lem:GBasic} and \cref{lem:GMeasurability} collectively confirm the range of $G_t^\lambda$ and $H^\fp_t$, as specified in their respective definitions.
\begin{lemma}\label{lem:GBasic}
The following is true for any $(t,\lambda)\in\bN\times\Lambda$:
\begin{enumerate}
\item[(a)] Suppose \cref{asmp:Main} (iv) holds. If $v^1,v^2\in\ell^\infty(\bX,\cB(\bX))$ and $v^1\le v^2$, then $G^{\lambda}_t v^1 \le G^{\lambda}_t v^2$. Moreover, if $\sigma_{t,x}$ is normalized for $x\in\bX$, then $G^{\lambda}_t v(x) \in \bigl[\,-b-\gamma\|v\|_\infty,\; b+\gamma\|v\|_\infty\,\bigr]$.
\item[(b)] For any $a\in\bR$ and $v\in\ell^\infty(\bX,\cB(\bX))$, we have $G^{\lambda}_t (a+v) = a + G^{\lambda}_tv$.
\end{enumerate}
Properties for $H_0$ hold analogously.
\end{lemma}
\begin{proof}
These are immediate consequences of \cref{def:DCRM}.
\end{proof}

\begin{lemma}\label{lem:GMeasurability}
Suppose  \cref{asmp:Main} (i) (iii) (iv) holds. For any $v\in\ell^\infty(\bX,\cB(\bX))$, the mapping $(x,\lambda)\mapsto G^{\lambda}_t v(x)$ is $\cB(\bX)\otimes\cE(\Lambda)$-$\cB(\bR)$ measurable. Consequently, the mapping $x\mapsto H^{\fp}_t v(x)$ is $\cB(\bX)$-$\cB(\bR)$ measurable.
\end{lemma}
\begin{proof}
Let $c=b+\|v\|_\infty$. Let $B\in\cB([-c,c])$. Next, by \cref{lem:IntfMeasurability} we have
\begin{align*}
(x,a)\mapsto \int_{\bX} \1_{B}\big(C_t(x,a,y)+v(y)\big) \;P(t,x,a,\dif y)
\end{align*}
is $\cB(\bX)\times\cB(\bA)$-$\cB(\bR)$ measurable. By \cref{lem:sigmaAlgBE}  (where we set  $f((x,\lambda),a)=\int_{\bX} \1_{B}\big(C_t(x,a,y)+v(y)\big)\;  P(t,x,a,\dif y)$ and $M(x,\lambda)=\lambda$),
\begin{align*}
(x,\lambda)\mapsto \int_{\bA}\int_{\bX} \1_{B}\big(C_t(x,a,y)+v(y)\big) P(t,x,a,\dif y)\lambda(\dif a)
\end{align*}
is $\cB(\bX)\otimes\cB(\Lambda)$-$\cB(\bR)$ measurable. Thus, $(x,\lambda)\mapsto \wt P^{v}_{t,x,\lambda}$ is $\cB(\bX)\otimes\cB(\Lambda)$-$\cB(\cP([-c,c]))$ measurable. Then, by the monotone class theorem (cf. \cite[Theorem 1.9.3 (ii)]{Bogachev2006book}), $(x,\lambda)\mapsto(x,\wt P^{v}_{t,x,\lambda})$ is $\cB(\bX)\otimes\cB(\Lambda)$-$\cB(\bX)\otimes\cB(\cP([-c,c]))$ measurable. The first claim of the theorem then follows as compositions of measurable functions are measurable (cf. \cite[Section 4.5, Lemma 4.22]{Aliprantis2006book}). The second statement follows analogously.
\end{proof}

Next, we reformulate finite horizon DRMs as compositions of $H^\fp_t$ operators. Let $O(x):=0$ for $x\in\bX$.
\begin{lemma}\label{lem:rhotTH}
Under Assumption \ref{asmp:Main} (iii) (iv), for any $0<t\le T<\infty$,
\begin{align*}
\varsigma^{\fX^\fp}_{t,X_t,T}
\big(\fC^\fp\big)
= H^{\fp}_t \circ\cdots\circ H^{\fp}_T O (X^{\fp}_t), \quad \bP-a.s.,
\end{align*}
and 
\begin{align*}
\varsigma^{\fX^\fp}_{0,T}
\big(\fC^\fp\big) 
= H_0\circ H^{\fp}_1 \circ\cdots\circ H^{\fp}_T O.
\end{align*}
\end{lemma}
\begin{proof}
It is straightforward to verify that $\wt P^O_{T,X^\fp_{T},\pi_T(X^\fp_{T})}$ is a regular conditional distribution of $C_T(X^\fp_{T},A^\fp_T,X^\fp_{T+1})$ given $\sU_T$. By \eqref{eq:DefsigmatT} and \eqref{eq:DefH}, we have
\begin{align*}
\varsigma^{\fX^\fp}_{T,X^\fp_t,T}\big(\fC^\fp\big) 
= \sigma_{T,X^\fp_T}\big(\wt P^O_{T,X^\fp_{T},\pi_T(X^\fp_{T})}\big) = H^{\fp}_T O (X^{\fp}_T), \quad \bP-a.s.
\end{align*}
Note that $\wt P^{ H^{\fp}_T O}_{T,X^\fp_{T-1},\pi_{T-1}(X^\fp_{T-1})}$ is the regular conditional distribution of $C_{T-1}\big(X^\fp_{T-1},A^\fp_{T-1},X^\fp_{T}\big)
\allowbreak
+
\gamma \,\varsigma^{\fX^\fp}_{T,X^\fp_t,T}\big(\fC^\fp\big)$ 
given $\sU_{T-1}$. By \eqref{eq:DefsigmatT} and \eqref{eq:DefH} again, we have
\begin{align*}
\varsigma^{\fX^\fp}_{T-1,X_{T-1}, T}\big(\fC^\fp\big) = \sigma_{T-1,X_{T-1}}\big(\wt P^{ H^{\fp}_T O}_{T-1,X^\fp_{T-1},\pi_{T-1}(X^\fp_{T-1})}\big) = H^{\fp}_{T-1}\circ H^{\fp}_T O (X^{\fp}_{T-1}),\quad\bP-a.s.
\end{align*}
Proceeding by induction backwards, we conclude the proof. 
\end{proof}

Below, we show that the mapping $(x,\lambda)\mapsto G_t^\lambda v(x)$ inherits the lower semi-continuity property from $v$.
\begin{lemma}\label{lem:GvLSC}
Let $(x^n)_{n\in\bN}\subset\bX$ and $(\lambda^n)_{n\in\bN}\subset\Lambda$ converge to $x^0\in\bX$ and $\lambda^0\in\Lambda$, respectively. Under \cref{asmp:Main} (i) (iii) (iv), if $v\in\ell^\infty(\bX,\cB(\bX))$ is lower semi-continuous, then 
\begin{align*}
\liminf_{n\to\infty} G^{\lambda^n}_t v (x^n) \ge G^{\lambda^0}_t v(x^0).
\end{align*}
\end{lemma}
\begin{proof}

We proceed by contradiction. Suppose $\liminf_{n\to\infty} G^{\lambda^n}_t v (x^n)<G^{\lambda^0}_t v(x^0)$. By definition (see \eqref{eq:DefG}), we have
\begin{align}\label{eq:GLSCContradiction}
\liminf_{n\to\infty} \sigma_{t,x^n}\left(\wt P^v_{t,x^n,\lambda^n}\right) < \sigma_{t,x^0}\left(\wt P^v_{t,x^0,\lambda^0}\right). %\le \liminf_{n\to\infty} \sigma_{t,x^n}(\wt P^v_{t,x^0,\lambda^0}).
\end{align}
Thus, there exists a subsequence $(x^{n_k})_{k\in\bN}$ such that
\begin{align*}
\lim_{k\to\infty}\sigma_{t,x^{n_k}}\left(\wt P^v_{t,x^{n_k},\lambda^{n_k}}\right) = \liminf_{n\to\infty} \sigma_{t,x^n}\left(\wt P^v_{t,x^n,\lambda^n}\right)\,.
\end{align*}
To ease the notation, without loss of generality, we assume  $\big(\sigma_{t,x^{n}}\left(\wt P^v_{t,x^{n},\lambda^{n}}\right)\big)_{k\in\bN}$ converges. Next, let $c=b+\gamma\,\|v\|_\infty$ and note that $\supp\wt P^v_{t,x^n,\lambda^n}\subseteq[-c,c]$ for $n\in\bN$ due to \cref{asmp:Main} (iv). Thus, by the Prokhorov theorem (cf. \cite[Section 15.6, Theorem 15.22]{Aliprantis2006book}), $(\wt P^v_{t,x^n,\lambda^n})_{n\in\bN}$ has a weakly converging subsequence. Thus, without loss of generality, we assume $(\wt P^v_{t,x^n,\lambda^n})_{n\in\bN}$ converges weakly and denote its limit by $\overline P^0$. By \cref{lem:Portmanteau} (f) and the fact that a CDF has at most countably many points of discontinuity, we have 
\begin{align}\label{eq:rConvDense}
\left\{r\in\bR: \lim_{n\to\infty}\wt P^v_{t,x^n,\lambda^n}((r,\infty)) = \overline P^0((r,\infty)) \right\} \text{ is dense in } \bR.
\end{align}
Moreover, note that $(x,a,y)\mapsto\1_{(r,\infty)}\big(C_{t}(x,a,y)+\gamma v(y)\big)$ is lower semicontinuous. In view of \cref{asmp:Main} (i), applying \cref{lem:ConvVaryingMeas} repeatedly, we obtain 
\begin{align}%\label{eq:GLSCContradiction3}
    \liminf_{n\to\infty}\wt P^v_{t,x^n,\lambda^n}((r,\infty)) &= \liminf_{n\to\infty} \int_{\bA}\int_{\bX} \1_{(r,\infty)}\big(C_t(x^n,a,y)+\gamma v(y)\big) P(t,x^n,a,\dif y)\lambda^n(\dif a) 
    \nonumber 
    \\
    &\ge \int_{\bA}\int_{\bX} \1_{(r,\infty)}(C_t(x^0,a,y)+\gamma v(y)) P(t,x^0,a,\dif y)\lambda^0(\dif a) \nonumber 
    \\
    &= \wt P^v_{t,x^0,\lambda^0}((r,\infty)).    
\end{align}
Therefore, $\limsup_{n\to\infty}\wt P^v_{t,x^n,\lambda^n}((-\infty,r]) \le \wt P^v_{t,x^0,\lambda^0}((-\infty,r])$ for $r\in\bR$. This together with \eqref{eq:rConvDense} and the right contintuity of CDF implies
\begin{align*}
\overline P^0((-\infty,r]) \le  \wt P^v_{t,x^0,\lambda^0}((-\infty,r]), \quad r\in\bR.
\end{align*}
It follows from \cref{asmp:Main} (iii) and \cref{def:DCRM} (ii) that
\begin{align*}
\liminf_{n\to\infty}\sigma_{t,x^n}\left(\wt P^v_{t,x^n,\lambda^n}\right) \ge \sigma_{t,x^0}\left(\overline P^0\right) \ge \sigma_{t,x^0}(\wt P^v_{t,x^0,\lambda^0}),
\end{align*}
which contradicts \eqref{eq:GLSCContradiction}.
\end{proof}

Next, we introduce a contraction property below. This contraction property  proves useful when we establish the infinite horizon DPP in \cref{sec:DPPInfinite}.
\begin{lemma}\label{lem:GContr}
Suppose \cref{asmp:Main} (iv) holds. For any $t\in\bN$, $\lambda\in\Lambda$ and $v^1,v^2\in\ell^\infty(\bX,\allowbreak\cB(\bX))$, we have 
\begin{align*}
\left\|G^{\lambda}_t v^1 - G^{\lambda}_t v^2\right\|_\infty \le \gamma\,\|v^1-v^2\|_\infty
\end{align*}
and
\begin{align*}
\left|H_0 v^1 - H_0 v^2\right| \le \gamma\,\|v^1-v^2\|_\infty.
\end{align*}
\end{lemma}
\begin{proof}
Note that $v^1-\|v^1-v^2\|_\infty \le v^2 \le v^1+\|v^1-v^2\|_\infty$. In view of \eqref{eq:DefPtilde}, for any $r\in\bR$, we have 
\begin{align*}
\wt P^{v^1-\|v^1-v^2\|_\infty}_{t,x,\lambda}((-\infty,r]) \ge \wt P^{v^2}_{t,x,\lambda}((-\infty,r]) \ge \wt P^{v^1+\|v^1-v^2\|_\infty}_{t,x,\lambda}((-\infty,r]).
\end{align*}
The above together with \cref{def:DCRM} (ii) implies that 
\begin{align*}
\sigma_{t,x}\left(\wt P^{v^1-\|v^1-v^2\|_\infty}_{t,x,\lambda}\right) \le \sigma_{t,x}\left(\wt P^{v^2}_{t,x,\lambda}\right) \le \sigma_{t,x}\left(\wt P^{v^1+\|v^1-v^2\|_\infty}_{t,x,\lambda}\right),\quad x\in\bX.
\end{align*}
It follows from \cref{def:DCRM} (i) and \eqref{eq:DefPtilde} that 
\begin{align*}
\sigma_{t,x}\left(\wt P^{v^1}_{t,x,\lambda}\right) - \gamma\|v^1-v^2\|_\infty \le \sigma_{t,x}\left(\wt P^{v^2}_{t,x,\lambda}\right) \le \sigma_{t,x}\left(\wt P^{v^1}_{t,x,\lambda}\right) + \gamma\|v^1-v^2\|_\infty,\quad x\in\bX.
\end{align*}
Hence, recalling the definition of $G_t^\lambda$ from \eqref{eq:DefG}, this proves the inequality with $G^\lambda_t$. The inequality with $H_0$ can be proved analogously.
\end{proof}

\section{Finite horizon dynamic programming principle}\label{sec:DPPFinite}

To obtain our main result concerning the finite horizon DPP, we require proving regularity of its various components. To do so,  we introduce some new notations and some technical lemmas.

For $T\in\bN$ and $t\in\set{0,1,\dots,T}$ we use a recursion to define the value functions as follows
\begin{align}\label{eq:DefJtTp}
J^{\fp}_{t,T} := \begin{cases}
O\,, & t > T,
\\
H^{\fp}_t J^{\fp}_{t+1,T}\,, & t \le T,
\end{cases}
\end{align}
and  $J^{\fp}_{0,T}:= H_0 J^{\fp}_{1,T}$.
Note that, by \cref{lem:GMeasurability}, the mapping $x\mapsto H^{\fp}_tv(x)$ is $\cB(\bX)$-$\cB(\bR)$ measurable. As a direct consequence of \cref{lem:rhotTH}, we have the following relationship between DRMs and value functions:
\begin{align}\label{eq:sigmaJ}
\varsigma^{\fX^\fp}_{0,T}(\fC^\fp) = J^{\fp}_{0,T},\quad\text{and}\quad \varsigma^{\fX^\fp}_{t,\fX^\fp,T}(\fC^\fp) =  J^{\fp}_{t,T}(X^{\fp}_t),\qquad
\bP\text{-a.s.},\, t,T\in\bN,\, t<T.
\end{align}
We next define the risk-averse version of a Bellman operator, denoted $S_t$, which acts on $v\in\ell^\infty(\bX,\cB(\bX))$, as
\begin{align}\label{eq:DefS}
S_t v(x) := \inf_{\lambda\in\varpi_t(x)}G^{\lambda}_t v(x) = \inf_{\lambda\in\varpi_t(x)}\sigma_{t,x}\left(\wt P^v_{t,x,\lambda}\right),\qquad t\in\bN.
\end{align}
%We also define $\fS \fv := (S_t v_{t+1})_{t\in\bN}$. 
For $t,T\in\bN$, we define 
\begin{align}\label{eq:DefJstartT} 
J^{*}_{t,T} := 
\begin{cases}
O, & t > T\\
S_{t} J^{*}_{t+1,T}, & t \le T
\end{cases}
\end{align}
and $J^*_{0,T}:=H_0 J^*_{1,T}$. For definition \eqref{eq:DefJstartT}  to be well defined, we must investigate the measurability of $J_{t+1,T}^*$ 
(boundedness is obvious from \cref{lem:GBasic} (a)). The lemma below resolves this issue by establishing its lower semicontinuity.
\begin{lemma}\label{lem:JtTLSC}
Under \cref{asmp:Main}, for any $t,T\in\bN$ with $t<T$, $x\mapsto J^{*}_{t,T}(x)$ is lower semi-continuous on $\bX$. 
\end{lemma}
\begin{proof}
Due to Assumption \ref{asmp:Main} (i) and Prokhorov theorem (cf. \cite[Section 15.6, Theorem 15.22]{Aliprantis2006book}), $\varpi_t(x)$ is compact. If $t=T$, then $J^{*}_{t,T} = S_tO$ by \eqref{eq:DefJstartT}, and the lower semi-continuity in $x\in\bX$ follows from \cref{asmp:Main} (ii), \cref{lem:GvLSC}, and \cref{lem:InffLSC}. We induce backward for $t<T$. As $J^*_{t,T}=S_tJ^{*}_{t+1,T}$, invoking \cref{asmp:Main} (ii), \cref{lem:GvLSC}, and \cref{lem:InffLSC} again, we conclude the proof.
\end{proof}

Below we present one of our main results -- a finite horizon DPP.
\begin{theorem}\label{thm:FiniteDPP}
Let $\gamma\in[0,1]$ and $T\in\bN$. Under Assumption \ref{asmp:Main}, the following is true.
\begin{enumerate}
\item[(a)] For any $t\in\set{0,\dots,T}$ and $x\in\bX$, $\argmin_{\lambda\in\varpi_t(x)}G^{\lambda}_t J^{*}_{t,T}(x)$ is not empty and closed, and there is a measurable $\pi^*_t:(\bX,\cB(\bX))\to(\Lambda,\cE(\Lambda))$ such that 
\begin{align}\label{eq:Finitepistar}
\pi^*_t(x)\in\argmin_{\lambda\in\varpi_t(x)}G^{\lambda}_t J^{*}_{t+1,\infty}(x),\quad x\in\bX.
\end{align}
\item[(b)] %Let $\pi^*_t$ be any mapping from $\bX$ to $\Lambda$ such that $\supp\pi^*_t(x)\subseteq \argmin_{\lambda\in\Lambda}G^{\lambda,u}_t J^{*,\fu}_{t+1,\infty}(x)$ for any $x\in\bX$ and $t\in\bN$. 
For $\fp^*=(\pi^*_t)_{t\in\set{0,\dots,T}}$ satisfying \eqref{eq:Finitepistar} for all $t\in\set{0,\dots,T}$, we have $J^{\fp^*}_{t,T}=J^{*}_{t,T}=\inf_{\fp\in\Pi}J^{\fp}_{t,T}$ for $0<t<T$, and $J^{\fp^*}_{0,T} = J^{*}_{0,T} = \inf_{(\fX^\fp,\fA^\fp)\in\Psi}\rho^{\fX^\fp}_{0,T}(\fC^\fp),$
i.e., $\fp^*$ is the optimal Markovian policy.
\end{enumerate}
\end{theorem}
\begin{proof}
\textbf{(a)} 
We fix $t\in\bN$ for the rest of the proof. By \cref{lem:GvLSC} and \cref{lem:JtTLSC}, $(x,\lambda) \mapsto G^{\lambda}_t J^{*}_{t+1,T}(x)$ is lower semi-continuous. Due to \cref{asmp:Main} (ii), $\varpi_t(x)$ is compact. It follows that for $x\in\bX$, $\argmin_{\lambda\in\varpi_t(x)}G^{\lambda}_t J^{*}_{t+1,T}(x)$ is not empty and closed. We claim that the lower semi-continuity of $(x,\lambda) \mapsto G^{\lambda}_t J^{*}_{t+1,T}(x)$ implies that for any closed $F\subseteq\bigcup_{x\in\bX}\varpi_t(x)$,
\begin{align*}%\label{eq:BorelXMeas}
B_F := \left\{x\in\bX: \;\argmin_{\lambda\in\varpi_t(x)}G^{\lambda}_t J^{*}_{t+1,T}(x)\cap F \ne \emptyset \right\} \in \cB(\bX),
\end{align*}
i.e., the set-valued mapping $x\mapsto\argmin_{\lambda\in\varpi_t(x)}G^{\lambda}_t J^{*}_{t+1,T}(x)$ is $\cB(\bX)$-measurable (cf. \cite[Section 18.1, Definition 18.1]{Aliprantis2006book}). To this end note that by statement (a) we have
\begin{align*}
B_{F} = \left\{x\in\bX: J^{*}_{t,T}(x) = \min_{\lambda\in F}G^{\lambda}_t J^{*}_{t+1,T}(x)\right\}.
\end{align*}
Above, we note that $x\mapsto\min_{\lambda\in F}G^{\lambda}_t J^{*}_{t+1,T}(x)$ is well defined and lower semi-continuous in $x\in\bX$, due to \cref{lem:GvLSC}, \cref{lem:JtTLSC}, \cref{lem:InffLSC}, and the fact that $F\subseteq\bigcup_{x\in\bX}\varpi_t(x)$ is compact and upper hemi-continuous as a constant set-valued function. By \cref{lem:JstarInftyLSC} again that $J^{*}_{t,T}$ is lower semi-continuous, we yield that both $\min_{\lambda\in F}G^{\lambda}_t J^{*}_{t+1,T}$ and $J^*_{t,T}$ are $\cB(\bX)$-$\cB(\bR)$ measurable. Consequently, $B_{F}\in\cB(\bX)$. Recall that a measurable set-valued function is also weakly measurable (cf. \cite[Section 18.1, Lemma 18.2]{Aliprantis2006book}). Then, by  applying the Kuratowski and Ryll-Nardzewski measurable selection theorem (cf. \cite[Section 18.3, Theorem 18.13]{Aliprantis2006book}), the set-valued function $x\mapsto\argmin_{\lambda\in\varpi_t(x)}G^{\lambda}_t J^{*}_{t+1,T}(x) $ has a $\cB(\bX)$-$\cB(\Lambda)$ measurable selector. In view of \cref{lem:sigmaAlgBE}, such a selector is also $\cB(\bX)$-$\cE(\Lambda)$ measurable.

\textbf{(b)} Combining Lemma \ref{lem:GBasic} (a) with \eqref{eq:DefJtTp} and \eqref{eq:DefJstartT}, we obatin $J^{*}_{t,T}(x) \le J^{\fp}_{t,T}(x)$ for any $t\le T$, $x\in\bX$ and $\fp\in\Pi$. It follows from induction that, regardless of $\fp$, 
\begin{align*}
J^*_{0,T}\le J^\fp_{0,T}=\varsigma^{\fX^\fp}_{0,T}(\fC^\fp),
\end{align*}
where we use \eqref{lem:rhotTH} to obtain the second equality. Next,  by \eqref{eq:DefJstartT} and statement (a) we have
\begin{align*}
J^*_{t,T} = S_{t} J^{*}_{t+1,T} = G^{\pi^*_t(x)}_t J^{*}_{t+1,T}(x) = H_t^{\fp^*}J^*_{t+1,T}(x).
\end{align*} 
Inducing backward from $t=T$, in view of \eqref{eq:DefJtTp}, we have $J^*_{t,T}=H^{\fp^*}_t\circ\cdots\circ H^{\fp^*}_T O=J^{\fp^*}_{t,T}$ and $J^*_{0,T}=H_0\circ H^{\fp^*}_1 \circ\cdots\circ H^{\fp^*}_T O=J^*_{0,T}$. Finally, by \cref{lem:rhotTH}, we conclude the proof.
\end{proof}

As the last result in this section, we argue that $\fp^*$ introduced in \cref{thm:FiniteDPP} (b) is no worse than any other history-dependent policy.
\begin{proposition}\label{prop:MarkovControlOptFinite}
Let $\gamma\in[0,1]$ and $T\in\bN$. Under \cref{asmp:Main}, we have
$$\inf_{(\fX^\fp,\fA^\fp)\in\Psi}\varsigma^{\fX^\fp}_{0,T}(\fC^\fp) \le \varsigma^{\fX}_{0,T}(\fC) \quad\text{for any}\quad (\fX,\fA)\in\Psi.$$
\end{proposition}
\begin{proof}
Fix $(\fX,\fA)\in\Psi$ for the remainder of the proof. For $t\in\bN$, let $P^{A_t|\sU_t}$ be the regular version of $\bP(A_t\in\,\cdot\,|\sU_t)$. It is straightforward to verify that 
$$\int_{\bA} \int_{\bX} \1_\cdot(C_T(X_T,a,y)) P(T,X_T,a,\dif y)\,P^{A_T|\sU_{T}}(\dif a)$$ 
is a regular version of $\bP\big(C_t(X_T,A_T,X_{T+1})\in\,\cdot\,|\sU_T\big)$.  This together with \eqref{eq:DefsigmatT} implies that
\begin{align*}
\sigma^\fX_{T,x,T}(\fC) = \sigma_{T,x}\left(\int_{\bA} \int_{\bX} \1_\cdot(C_T(X_T,a,y)) P(T,X_T,a,\dif y)\,P^{A_T|\sU_{T}}(\dif a)\right)\;.
\end{align*}
Due to the definition of $\Psi$ in \cref{subsec:Problem}, we have that 
\[
P^{A_T|\sU_T}(B)|_{B=\cA_T(X_T)}=\int_{\bA}\1_{\cA_T(X_T)}(a)\, P^{A_T|\sU_T}(\dif a)=1,\;\bP-a.s.
\]
%and Lemma \ref{lem:CondExpnXY}
(we also use \cite[Section 18.1, Theorem 18.6]{Aliprantis2006book} for the joint measurability of $\1_{\cA_T(x)}(a)$ as a function of $(x,a)$). It follows from \eqref{eq:DefG} and \eqref{eq:DefJstartT}  that $\sigma^\fX_{T,X_T,T}(\fC)\ge J^*_{T,T}(X_T),\,\bP$-a.s.  We next proceed to pull back the time index by induction. Suppose for some $t\in\set{1,\dots,T-1}$, we have $\varsigma^\fX_{t+1,X_{t+1},T}(\fC)\ge J^*_{t+1,T}(X_{t+1}),\,\bP$-a.s. Then, by \eqref{eq:DefsigmatT} and a similar reasoning leading to \cref{prop:TimeConsistency}, we obtain 
$$\varsigma^\fX_{t,X_t,T}(\fC) = \sigma_{t,x}\left(P^{C_{t}(X_t,A_t,X_{t+1}) + \gamma\;\varsigma^{\fX,\gamma}_{t+1,X_{t+1},T}(\fZ)|\sU_t}\right) \ge \sigma_{t,X_t}\left(P^{C_{t}(X_t,A_t,X_{t+1})+\gamma J^*_{t+1,T}(X_{t+1})|\sU_t}\right), \quad\! \bP\text{-a.s.}$$ 
Note that 
$$\int_{\bA} \int_{\bX} \1_\cdot(C_t(X_t,a,y)+\gamma J^*_{t+1,T}(y)) P(t,X_t,a,\dif y)\,P^{A_t|\sU_t}(\dif a)$$
is a regular version of $P^{C_{t}(X_t,A_t,X_{t+1})+\gamma J^*_{t+1,T}(X_{t+1})|\sU_t}$. 
By the fact that $P^{A_t|\sU_t}(B)|_{B=\cA_t(X_t)}=1,\,\bP$-a.s., \eqref{eq:DefG} and \eqref{eq:DefJstartT} again, we have $\varsigma^\fX_{t,X_t,T}(\fC)\ge J^*_{t,T}(X_{t})$, $\bP$-a.s. As a result of the induction above, $\varsigma^\fX_{1,X_1,T}(\fC)\ge J^*_{1,T}(X_{1}),\,\bP$-a.s.  With similar reasoning, we obtain $\varsigma^\fX_{0,T}(\fC)\ge J^*_{0,T}$. Finally, invoking Theorem \ref{thm:FiniteDPP} (b) completes the proof.
\end{proof}

\section{Infinite horizon dynamic programming principle}\label{sec:DPPInfinite}

As discussed in \cref{subsec:DDRM}, we make the additional assumptions that $\gamma\in[0,1)$ and $(\sigma_{t,x})_{(t,x)\in\bN\times\bX}$ is normalized to allow the infinite horizon object $\varsigma_{t,x,\infty}^{\fX}$ to be well defined. We adhere to these assumptions throughout this section. Recall that $(\sigma_{t,x})_{(t,x)\in\bN\times\bX}$ is normalized if $\sigma_{t,x}$ is normalized for $(t,x)\in\bN\times\bX$.

To establish an infinite horizon DPP for \eqref{eq:main}, we first study the value functions associated with a Markovian policy $\fp$. Recall the definitions of $G_{t}, H^{\fp}_{t}$ and $J^*_{t,T}$ from \eqref{eq:DefG}, \eqref{eq:DefH}, and \eqref{eq:DefJstartT}, respectively. We justify the definition of 
\begin{align}\label{eq:DeftInftyP}
J^{\fp}_{0,\infty}:=\lim_{T\to\infty}J^{\fp}_{0,T} \quad\text{and}\quad J^{\fp}_{t,\infty}(x):=\lim_{T\to\infty}J^{\fp}_{t,T}(x),\quad x\in\bX
\end{align}
in \cref{lem:JpInftyBasic} below, which in term reveals the relationship between $J^{\fp}_{t,\infty}$ and $J^{\fp}_{t+1,\infty}$. Moreover, \cref{lem:JpInftyBasic} together with \cref{lem:sigmatTConv} implies that
\begin{align}\label{eq:rhoJEquiv}
\sigma^{\fX^\fp}_{0,\infty}(\fC^\fp) = J^{\fp}_{0,\infty}\quad\text{and}\quad\sigma^{\fX^\fp}_{t,X_t,\infty}(\fC^\fp) = J^{\fp}_{t,\infty}(X^{\fp}_t),\quad\bP\text{-a.s.},\, t\in\bN.
\end{align}
\begin{lemma}\label{lem:JpInftyBasic}
Suppose that $\gamma\in[0,1)$ and $(\sigma_{t,x})_{(t,x)\in\bN\times\bX}$ is normalized. Under Assumption \ref{asmp:Main} (iv), for any $\fp\in\Pi$, we have $(J^\fp_{0,T})_{T\in\bN}$ converges and $(J^{\fp}_{t,T})_{T\in\set{t,t+1,\dots}}$ converges uniformly for $t\in\bN$. Moreover,
\begin{align}\label{eq:JHJ}
J^{\fp}_{t,\infty} := \lim_{T\to\infty}J^{\fp}_{t,T} = \begin{cases}
H^{\fp}_t J^{\fp}_{t+1,\infty}, & t\in\bN,
\\[0.5em]
H_0 J^{\fp}_{1,\infty}, & t = 0.
\end{cases}
\end{align}
Lastly, $\sup_{t\in\bN}\|J^\fp_{t,\infty}\|_\infty\le\frac{b}{1-\gamma}$.
\end{lemma}
\begin{proof} 
Let $r\in\bN$. First, by \cref{asmp:Main} (iv) and \cref{lem:GBasic} (a), we have 
\begin{align*}
H^\fp_{T+r}O(x) = G^{\pi_{T+r}(x)}_{T+r}O(x) \in [-b,b],\quad x\in\bX,
\end{align*}
i.e., $\|H^{\fp}_{T+r} O\|_\infty \le b$. Then, by \cref{asmp:Main} (iv) and \cref{lem:GBasic} (a),
\begin{align*}
H^{\fp}_{T+r-1}\circ H^{\fp}_{T+r} O(x) = G^{\pi_{T+r-1}(x)}_{T+r-1} \circ H^{\fp}_{T+r} O (x) \in [-b-\gamma b, \;b+\gamma b],
\end{align*}
i.e., $\|H^{\fp}_{T+r-1}\circ H^{\fp}_{T+r} O\|_\infty \le b + \gamma b$. By induction, we have $\|H^{\fp}_{T+1}\circ\cdots\circ H^{\fp}_{T+r} O\|_\infty \le (1-\gamma)^{-1} b$. Then, by \cref{lem:GContr} we obtain
\begin{align*}
\left|H^{\fp}_T O(x) - H^{\fp}_{T}\circ \cdots \circ H^{\fp}_{T+r} O(x)\right| &= \left|G^{\pi_T(x)}_TO(x) - G^{\pi_T(x)}_T \circ H^{\fp}_{T+1}\circ \cdots \circ H^{\fp}_{T+r} O (x)\right|
\\
&\le \gamma\|O-H^{\fp}_{T+1}\circ \cdots \circ H^{\fp}_{T+r}O\|_\infty 
\\
&\le \frac{\gamma}{1-\gamma} b, \quad x\in\bX,
\end{align*}
which implies that $\left\| H^{\fp}_T O - H^{\fp}_{T}\circ \cdots \circ H^{\fp}_{T+r} O\right\|_\infty \le \frac{\gamma}{1-\gamma} b$. Inducing backward with \cref{lem:GContr}, we obtain that
\begin{align*}
&\left\|H^{\fp}_t \circ\cdots\circ H^{\fp}_T O - H^{\fp}_t \circ\cdots\circ H^{\fp}_{T} \circ H^{\fp}_{T+r} O\right\|_\infty\\
&\quad \le \left|G^{\pi_T(x)}_T\circ H^{\fp}_{t+1} \circ\cdots\circ H^{\fp}_T O (x) - G^{\pi_T(x)}_T \circ H^{\fp}_{t+1}\circ \cdots \circ H^{\fp}_{T+r} O (x)\right|\\
&\quad\le \gamma\left\|H^{\fp}_{t+1} \circ\cdots\circ H^{\fp}_T O (x) - G^{\pi_T(x)}_T\circ H^{\fp}_{t+1}\circ \cdots \circ H^{\fp}_{T+r} O \right\|_\infty \le \frac{\gamma^{T-t}}{1-\gamma} b.
\end{align*}
The above proves that $(J^{\fp}_{t,T})_{T\in\bN}$ converges uniformly. It also proves that $\sup_{t\in\bN}\|J^\fp_{t,\infty}\|\le\frac{b}{1-\gamma}$. 

Lastly, we prove \eqref{eq:JHJ}. For this, consider $t\in\bN$. In view of \eqref{eq:DefJtTp}, \cref{lem:GContr} and the uniform convergence proved above, we have
\begin{align*}
\|J^{\fp}_{t,\infty} - H^{\fp}_t J^{\fp}_{t+1,\infty}\|_\infty &\le \|J^{\fp}_{t,\infty} - J^{\fp}_{t,T}\|_\infty + \|J^{\fp}_{t,T}-H^{\fp}_t J^{\fp}_{t+1,T}\|_\infty + \|H^{\fp}_t J^{\fp}_{t+1,T}-H^{\fp}_t J^{\fp}_{t+1,\infty}\|_\infty\\
& \le  (1+\gamma) \|J^{\fp}_{t,\infty}-J^{\fp}_{t,T}\|_\infty \xrightarrow[T\to\infty]{} 0.
\end{align*}
This proves the case of $t\in\bN$ in \eqref{eq:JHJ}. The case of $t=0$ can be proved similarly.
\end{proof}

Next, we introduce below a proposition regarding the value of a Markovian policy $\fp$.
First, let $\ell^\infty(\bN;\bX,\cB(\bX))$ be the set of $\fv=(v_t)_{t\in\bN}$ such that $v_t:(\bX,\cB(\bR))\to(\bR,\cB(\bR))$ for $t\in\bN$ and $\|\fv\|_\infty:=\sup_{t\in\bN}\|v_t\|_\infty<\infty$. For $\fv\in\ell^\infty(\bN;\bX,\cB(\bX))$, we define $\fH^{\fp} \fv := (H^{\fp}_t v_{t+1})_{t\in\bN}$.  We then have the following.
\begin{proposition}\label{prop:PolicyEval}
Suppose that $\gamma\in[0,1)$ and $(\sigma_{t,x})_{(t,x)\in\bN\times\bX}$ is normalized. Under \cref{asmp:Main} (iv), for any Markovian policy $\fp\in\Pi$, $\fH^{\fp}$ is a $\gamma$-contraction mapping on $\ell^\infty(\bN;\bX,\cB(\bX))$. Moreover, $(J^{\fp}_{t,\infty})_{t\in\bN}$ is the unique fixed point of $\fH^{\fp}$.
\end{proposition}
\begin{proof}
In view of \cref{lem:GBasic} and \cref{lem:GMeasurability}, $\fH^\fp$ indeed maps from $\ell^\infty(\bN;\bX,\cB(\bX))$ to $\ell^\infty(\bN;\bX,\cB(\bX))$, and it is a $\gamma$-contraction due to \cref{lem:GContr}.  As pointwise convergence preserves measurability (cf. \cite[Section 4.6, Lemma 4.29]{Aliprantis2006book}), by combining \cref{lem:GMeasurability} with the convergence and the bound established in \cref{lem:JpInftyBasic}, we have $(J^{\fp}_{t,\infty})_{t\in\bN}\in\ell^\infty(\bN;\bX,\cB(\bX))$. It follows from \eqref{eq:JHJ} that $(J^{\fp}_{t,\infty})_{t\in\bN}$ is a fixed point of $\fH^\fp$.
\end{proof}

To establish the infinite horizon DPP, we define the infinite horizon version of the optimal value function as 
\begin{align}\label{eq:DefJStartInfty}
J^*_{0,\infty}:=\lim_{T\to\infty}J^*_{0,T},\quad\text{and}\quad J^{*}_{t,\infty}(x):=\lim_{T\to\infty}J^{*}_{t,T}(x),\quad x\in\bX,\,t\in\bN.
\end{align}
The definition is justified by the following lemma.
\begin{lemma}\label{lem:JstarInftyUnifConv}
Suppose that $\gamma\in[0,1)$ and $(\sigma_{t,x})_{(t,x)\in\bN\times\bX}$ is normalized. Under Assumption \ref{asmp:Main}, for any $t\in\bN$, the sequence $(J^{*}_{t,T})_{T> t}$ converges uniformly and $\sup_{t\in\bN}\left\|J^{*}_{t,\infty}\right\|_\infty \le \frac{b}{1-\gamma}$. Moreover, $(J^*_{0,T})_{T\in\bN}$ converges and $J^*_{0,\infty} \le \frac{\gamma b}{1-\gamma}$.
\end{lemma}
\begin{proof}
Fix $t\in\bN$ for the remainder of this proof. First, we show that for any $t<T$, 
\begin{align}\label{eq:JBound}
\left\|J^{*}_{t,T}\right\|_\infty \le \frac{b}{1-\gamma}.
\end{align}
To this end, observe that $\left\|J^{*}_{T-1,T}\right\|_\infty\le b$ because $\|G^\lambda_{T-1}O\|_\infty\le b$ due to Assumption \ref{asmp:Main} (iv) and Lemma \ref{lem:GBasic} (a). By Assumption \ref{asmp:Main} (iv) and Lemma \ref{lem:GBasic} (a) again, $\left\|G^\lambda_{T-2}J^{*}_{T-1,T}\right\|_\infty\le b+\gamma b$ and thus $\left\|J^{*}_{T-2,T}\right\|_\infty \le b + \gamma b$. Then, \eqref{eq:JBound} follows by induction. Next, note for any $t\le T$, by Lemma \ref{lem:GContr},
\begin{align}\label{eq:DiffJIndc}
\left|J^{*}_{t,T}(x) - J^{*}_{t,T+r}(x)\right| &= \left|\inf_{\lambda\in\varpi_t(x)}G^{\lambda}_{t} J^{*}_{t+1,T}(x) - \inf_{\lambda\in\varpi_t(x)}G^{\lambda}_{t}J^{*}_{t+1,T+r}(x)\right|
\nonumber\\
&\le \sup_{\lambda\in\Lambda}\left|G^{\lambda}_{t} J^{*}_{t+1,T}(x) - G^{\lambda}_{t} J^{*}_{t+1,T+r}(x)\right| \le \gamma\left\|J^{*}_{t+1,T} - J^{*}_{t+1,T+r}\right\|_\infty.
\end{align}
In view of \eqref{eq:DefJstartT}, \eqref{eq:JBound} and \eqref{eq:DiffJIndc} together imply $\|J^{*}_{t,T} - J^{*}_{t,T+r}\|_\infty \le \frac{b}{1-\gamma}$ for $t=T$. Pulling $t$ backward with \eqref{eq:DiffJIndc}, we obtain
\begin{align*}
\left|J^{*}_{t,T}(x) - J^{*}_{t,T+r}(x)\right| \le \gamma\left\|J^{*}_{t+1,T} - J^{*}_{t+1,T+r}\right\|_\infty \le \gamma^{T-t}\|J^{*}_{T,T}(x) - J^{*}_{T,T+r}(x)\|_\infty \le \frac{\gamma^{T-t}b}{1-\gamma}.
\end{align*}
The above proves the uniform convergence of $(J^{*}_{t,T})_{T> t}$. By combining the uniform convergence with \eqref{eq:JBound}, we have that $\left\|J^{*}_{t,\infty}\right\|_\infty \le \frac{b}{1-\gamma}$. The proof for $J^*_{0,\infty}$ follows similarly.
\end{proof}

The following lower semi-continuity of $J^*_{t,\infty}$ is an immediate consequence of \cref{lem:JtTLSC} and \cref{lem:JstarInftyUnifConv}.
\begin{lemma}\label{lem:JstarInftyLSC}
Suppose that $\gamma\in[0,1)$ and $(\sigma_{t,x})_{(t,x)\in\bN\times\bX}$ is normalized. Under Assumption \ref{asmp:Main}, for any $t\in\bN$, $x\mapsto J^{*}_{t,\infty}(x)$ is lower semi-continuous on $\bX$. 
\end{lemma}

We are now in position to present the DPP for the infinite horizon optimization in \eqref{eq:main}. We recall the definition of $S_t$ from \eqref{eq:DefS}.
\begin{theorem}\label{thm:InfiniteDPP}
Suppose that $\gamma\in[0,1)$ and $(\sigma_{t,x})_{(t,x)\in\bN\times\bX}$ is normalized. Under Assumption \ref{asmp:Main}, the following is true.
\begin{enumerate}
\item[(a)] $(J^{*}_{t,\infty})_{t\in\bN}$ satisfies $J^*_{t,\infty}=S_t J^*_{t+1,\infty}$ for $t\in\bN$. Moreover, for $(J'_{t,\infty})_{t\in\bN}\in\ell^\infty(\bN;\bX,\cB(\bX))$ satisfying $J'_{t,\infty}=S_t J'_{t+1,\infty}$,
%\seb{I presume these should be $J'_{t,\infty}$}\ziteng{fixed} 
we have $(J'_{t,\infty})_{t\in\bN}=(J^*_{t,\infty})_{t\in\bN}$.
\item[(b)] For $t\in\bN$ and $x\in\bX$, $\argmin_{\lambda\in\varpi_t(x)}G^{\lambda}_t J^{*}_{t+1,\infty}(x)$ is not empty and closed, and there is a measurable $\pi^*_t:(\bX,\cB(\bX))\to(\Lambda,\cE(\Lambda))$ such that 
\begin{align}\label{eq:pistar}
\pi^*_t(x)\in\argmin_{\lambda\in\varpi_t(x)}G^{\lambda}_t J^{*}_{t+1,\infty}(x),\quad x\in\bX.
\end{align}
\item[(c)] %Let $\pi^*_t$ be any mapping from $\bX$ to $\Lambda$ such that $\supp\pi^*_t(x)\subseteq \argmin_{\lambda\in\Lambda}G^{\lambda,u}_t J^{*,\fu}_{t+1,\infty}(x)$ for any $x\in\bX$ and $t\in\bN$. 
For $\fp^*=(\pi^*_t)_{t\in\bN}$ satisfying \eqref{eq:pistar} for all $t\in\bN$, we have $J^{\fp^*}_{t,\infty} = J^{*}_{t,\infty} = \inf_{\fp\in\Pi} J^{\fp}_{t,\infty}$ for $t\in\bN$, and $J^{\fp^*}_{0,\infty} = H_0 J^{*}_{1,\infty} = \inf_{(\fX^\fp,\fA^\fp)\in\Psi}\rho^{\fX^\fp}_{0,\infty}(\fC^\fp)$.
\end{enumerate}
\end{theorem}
\begin{proof}
\textbf{(a)} First, observe that by \eqref{eq:DefJstartT}, Lemma \ref{lem:JstarInftyUnifConv} and \eqref{eq:SContr}, we have 
\begin{align*}
\left\|J^{*}_{t,\infty} - S_t J^{*}_{t+1,\infty}\right\|_\infty \le \left\|J^{*}_{t,\infty} -  J^{*}_{t,T}\right\|_\infty + 0 + \left\|S_tJ^{*}_{t+1,T} - S_t J^{*}_{t+1,\infty}\right\|_\infty \xrightarrow[T\to\infty]{} 0.
\end{align*}
The same is true for $(J'_{t,\infty})_{t\in\bN}$. Next, by Lemma \ref{lem:GContr}, we have
\begin{align}\label{eq:SContr}
\left|J^*_{t,\infty}(x) - J'_{t,\infty}(x)\right| &= \left|S_t J^*_{t+1,\infty}(x) - S_t J'_{t+1,\infty}(x)\right| = \left|\inf_{\lambda\in\varpi_t(x)}G^{\lambda}_t J^*_{t+1,\infty}(x) - \inf_{\lambda\in\varpi_t(x)}G^{\lambda}_t J'_{t+1,\infty}(x)\right|
\nonumber\\
&\le \sup_{\lambda\in\Lambda}\left|G^{\lambda}_t J^*_{t+1,\infty}(x) - G^{\lambda}_t J'_{t+1,\infty}(x)\right|
\nonumber
\\
&\le \gamma\;\|J^*_{t+1,\infty}-J'_{t+1,\infty}\|_\infty\,,
\qquad\qquad\qquad \forall\;(t,x)\in\bN\times\bX.
\end{align}
Therefore,
\begin{align}\label{eq:JstarJprime}
\sup_{t\in\bN}\|J^*_{t,\infty}-J'_{t,\infty}\|_\infty \le \gamma\sup_{t\in\bN}\|J^*_{t+1,\infty}-J'_{t+1,\infty}\|_\infty \le \gamma\sup_{t\in\bN}\|J^*_{t,\infty}-J'_{t,\infty}\|_\infty.
\end{align}
We note that $\sup_{t\in\bN}\|J^*_{t,\infty}-J'_{t,\infty}\|_\infty<\infty$ due to \cref{lem:JstarInftyUnifConv} and the hypothesis that $(J'_{t,\infty})_{t\in\bN}\in\ell^\infty(\bN;\bX,\cB(\bX))$. This together with \eqref{eq:JstarJprime} implies $\sup_{t\in\bN}\|J^*_{t,\infty}-J'_{t,\infty}\|_\infty=0$. The proof of statement (a) is complete. 

\textbf{(b)} 
%By Hypothesis \ref{hypo:main} (iii) and Lemma \ref{lem:JstarInftyLSC}, we have 
%$$h(x,a,q):=\int_{\bR\times\bX}\left(c + \gamma J^{*}_{t+1}(y)-q\right)_+ P^{a,u_t}(t,x,\dif c\,\dif y)$$ 
%is lower semi-continuous on $(x,a)\in\bX\times\bA$. Then, by Hypothesis \ref{hypo:main} (i) and Lemma \ref{lem:InffLSC}, we know 
By Lemma \ref{lem:GvLSC} and Lemma \ref{lem:JstarInftyLSC}, $(x,\lambda) \mapsto G^{\lambda}_t J^{*}_{t+1}(x)$ is lower semi-continuous. The remainder of the proof follows the same reasoning as in the proof of \cref{thm:FiniteDPP} (a).

\textbf{(c)} By \cref{lem:GBasic} (a), \eqref{eq:DefJtTp} and \eqref{eq:DefJstartT}, we obtain $J^{*}_{t,T}(x) \le J^{\fp}_{t,T}(x)$ for any $t\le T$, $x\in\bX$ and $\fp\in\Pi$. Then, by \cref{lem:JpInftyBasic} and \cref{lem:JstarInftyUnifConv}, we have $J^{*}_{t,\infty}(x) \le J^{\fp}_{t,\infty}(x)$ for any $t\in\bN$, $x\in\bX$ and $\fp\in\Pi$. On the other hand, note that by (a) and (b) we have
\begin{align*}
J^*_{t,\infty} = S_{t} J^{*}_{t+1,\infty} = G^{\pi^*_t(x)}_t J^{*}_{t+1,\infty}(x).
\end{align*}
It follows from Proposition \ref{prop:PolicyEval} that $J^{\fp^*}_{t,\infty} = J^{*}_{t,\infty}$ for $t\in\bN$. Consequently, we have that $J^{\fp^*}_{t,\infty} = J^{*}_{t,\infty} = \inf_{\fp\in\Pi} J^{\fp}_{t,\infty}$ for $t\in\bN$. Finally, this together with \eqref{eq:DefH0}, \eqref{eq:rhoJEquiv} and \eqref{eq:JHJ} implies $J^{\fp^*}_{0,\infty} = H_0 J^{*}_{1,\infty} = \inf_{\fp\in\Pi}\rho^{\fX^\fp}_{0,\infty}(\fC^\fp)$.
\end{proof}

\begin{remark}\label{rmk:Stationary}
When the transition kernel $P$, the action domain $\cA_t$, and the family of probability measures $\cM_t$ are constant in $t\in\bN$, it follows immediately from \eqref{eq:DefS} and \eqref{eq:DefJstartT} that $S_t$ is constant in $t\in\bN$ and $J^*_{t,T}=J^*_{t+1,T+1}$ for $t\in\bN$ and $t\le T$. Thus, by \cref{lem:JstarInftyUnifConv}, $J^*_t$ is also constant in $t$. The stationary versions of \cref{thm:InfiniteDPP} following immediately.
\end{remark}

To conclude this section, we argue that the Markovian $\fp^*$ introduced in \cref{thm:InfiniteDPP} (b) is also optimal among all history-dependent policies.
\begin{proposition}\label{prop:MarkovControlOpt}
Suppose that $\gamma\in[0,1)$ and $(\sigma_{t,x})_{(t,x)\in\bN\times\bX}$ is normalized. Under \cref{asmp:Main}, for any $(\fX,\fA)\in\Psi$ we have that
$$\inf_{(\fX^\fp,\fA^\fp)\in\Psi}\varsigma^{\fX^\fp}_{0,\infty}(\fC^\fp) \le \varsigma^{\fX}_{0,\infty}(\fC).$$
\end{proposition}
\begin{proof}
Fix $(\fX,\fA)\in\Psi$ for the remainder of the proof. In view of \cref{thm:InfiniteDPP} (c), it is sufficient to prove that $J^{*}_{0,\infty} \le \varsigma^{\fX}_{0,\infty}(\fC)$. For this, note that by \cref{thm:FiniteDPP} (b) and \cref{prop:MarkovControlOptFinite}, $J^*_{0,T} \le \varsigma^{\fX}_{0,T}(\fC)$. In view of \cref{lem:sigmatTConv} and \cref{lem:JstarInftyUnifConv}, letting $T\to\infty$, we conclude the proof.
\end{proof}

\section{Optimality of deterministic Markovian action}\label{sec:SuffDeter}

As a final key result, we provide a sufficient condition such that the optimal policies are deterministic mappings of state to action, which extends the discussion in \cite{Delage2019Dicesion} from the static case to dynamic case. The crux of this result is quite straightforward: if $\sigma_{t,x}$ penalizes uncertainty from mixtures\footnote{Indeed, the mixture $\theta\mu_1+(1-\theta)\mu_2$ arguably possesses a higher degree of uncertainty than one of $\mu_1$ or $\mu_2$. This can be intuitively understood by considering that $\theta\mu_1+(1-\theta)\mu_2$ corresponds to the probability of drawing from either $\mu_1$ or $\mu_2$, determined by the toss of a coin biased by the factor $\theta$.}, then deterministic actions may potentially reach the optimal outcome. For instance, when $\sigma_{t,x}$ stems from spectral risk measures, the conditions of \cref{prop:wpSingleton} are satisfied and the optimal policies are deterministic. For further details, we refer to \cref{rmk:CommentM}.
\begin{proposition}\label{prop:wpSingleton}
Let $t\in\bN$. Suppose \cref{asmp:Main} holds and
\begin{align*}
\sigma_{t,x}(\theta\mu_1+(1-\theta)\mu_2) \ge \min\left\{\sigma_{t,x}(\mu_1), \sigma_{t,x}(\mu_2)\right\}, \quad x\in\bX,
\end{align*}
for any $\mu_1,\mu_2\in\cP_b(\bR)$ and $\theta\in[0,1]$. Then, for any $T\in\bN$ with $T> t$, there is $\pi^\delta_t:(\bX,\cB(\bX))\to(\Lambda,\cE(\Lambda))$ such that $\pi^\delta_t(x)$ is a Dirac measure for all $x\in\bX$ and $$\pi^\delta_t(x)\in\argmin_{\lambda\in\varpi_t(x)}G^{\lambda}_t J^{*}_{t+1,T}(x).$$
If in addition, $\gamma\in(0,1)$ and that $(\sigma_{t,x})_{(t,x)\in\bN\times\bX}$ is normalized, then  the above statements hold for $T=\infty$.
\end{proposition}
\begin{proof}
We present the proof for the case of $T<\infty$, as proof for the case of $T=\infty$ follows along the same reasoning. To begin, note that, by \eqref{eq:DefG}, 
\begin{align}\label{eq:GlambdaIneq}
G^{\theta\lambda^1+(1-\theta)\lambda^2}_tv(x) &= \sigma_{t,x}\left(\wt P^v_{t,x,\theta\lambda_1+(1-\theta)\lambda_2}\right) = \sigma_{t,x}\left(\theta\wt P^v_{t,x,\lambda_1} + (1-\theta)\wt P^v_{t,x,\lambda_2}\right) \nonumber\\
&\ge \min\left\{\sigma_{t,x}\left(\wt P^v_{t,x,\lambda_1}\right), \sigma_{t,x}\left(\wt P^v_{t,x,\lambda_2}\right)\right\} = \min\left\{G^{\lambda^1}_tv(x), G^{\lambda^2}_tv(x)\right\}.
\end{align}
for any $v\in\ell^\infty(\bX,\cB(\bX))$, $\lambda^1,\lambda^2\in\Lambda$ and $\theta\in(0,1)$. 

It is sufficient to show that for any $x\in\bX$ the set 
$$D(x):=\set{a\in\cA_t(x):G^{\delta_a}_t J^{*}_{t+1,T}(x)=\min_{\lambda\in\varpi_t(x)}G^{\lambda}_t J^{*}_{t+1,T}(x)}$$ 
is not empty and closed. Once we demonstrate this, we can apply an argument similar to that used in the proof of Theorem \ref{thm:FiniteDPP} (a) to affirm the existence of $\pi^\delta_t$. %(by replacing $\argmin_{\lambda\in\varpi_t(x)}G^{\lambda}_t J^{*}_{t+1,\infty}(x)$ with $\argmin_{a\in\cA_t(x)}G^{\delta_a}_t J^{*}_{t+1,\infty}(x)$).

Let $\lambda^*\in\argmin_{\lambda\in\varpi_t(x)}G^{\lambda}_t J^{*}_{t+1,\infty}(x)$. For $A\in\cB(\bA)$, with the convention that $\frac 00=0$, we define 
\begin{align*}
\lambda^*_{A}:=\frac{\lambda^*(A^\varepsilon_k\cap\,\cdot\,)}{\lambda^*(A^\varepsilon_k)}.
\end{align*}
and claim that 
\begin{align}\label{eq:lambdaAargmin}
\left\{\lambda^*_A, \lambda^*_{A^c}\right\}\cap\argmin_{\lambda\in\varpi_t(x)}G^{\lambda}_t J^{*}_{t+1,T}(x) \neq \emptyset.
\end{align}
To see this, we note that if $\lambda^*_A=1$ or $\lambda^*_{A^c}=1$, we must have either $\lambda^*=\lambda^*_A$ or $\lambda^*=\lambda^*_{A^c}$, and \eqref{eq:lambdaAargmin} follows immediately. If none of $\lambda^*_A=1$ and $\lambda^*_{A^c}=1$ is true, then $\lambda^*_A$ and $\lambda^*_{A^c}$ are bona fide probabilities, and $\lambda^*=\lambda^*(A)\,\lambda^*_{A}+\lambda^*(A^c)\,\lambda^*_{A^c}$. This together with \eqref{eq:GlambdaIneq} proves \eqref{eq:lambdaAargmin}. 

Since $\cA_t(x)$ is compact (due to \cref{asmp:Main} (ii)) thus totally bounded (cf. \cite[Section 3.7, Theorem 3.28]{Aliprantis2006book}), i.e., for any $\varepsilon>0$ there exist finitely many $\varepsilon$-open balls that covers $\cA_t(x)$. %We define $B^\varepsilon_1=A^\varepsilon_1$ and $B^\varepsilon_k = A^\varepsilon_k\cap(B^\varepsilon_k)^c$ for $k>1$.
For $\varepsilon>0$, we consider $n_\varepsilon\in\bN$ and a partition $(A^\varepsilon_k)_{k=1}^{n_\varepsilon}$ of $\cA_t(x)$ such that the radius of $A^\varepsilon_k$ is no greater than $\varepsilon$. Invoking \eqref{eq:lambdaAargmin} repeatedly, there is $k=1,\dots,n_\varepsilon$ such that $\lambda^*_{A^\varepsilon_k}$ attains the optimal. Consequently, for any $m\in\bN$, there is $a^m\in\cA_t(x)$ and $\lambda^m\in\varpi_t(x)$ such that $\supp\lambda^m\subseteq \cA_t(x)\cap \overline B_{\frac1m}(a^m)$ and $\lambda^m\in\argmin_{\lambda\in\varpi_t(x)}G^{\lambda}_t J^{*}_{t+1,T}(x)$, where $\overline B_{\frac1m}(a^m)$ is the closed $\frac1m$-ball centered at $a^m$. As $\cA_t(x)$ is compact, without loss of generality, we assume $(a^m)_{m\in\bN}$ converges and set $a^0:=\lim_{m\to\infty}a^m$. Let $f\in C_b(\bA)$ and note $f$ is uniformly continuous on $\cA_t(x)$. Then, we have
\begin{align*}
\left|\int_\bA f(a)\lambda^m(\dif a) - f(a^0)\right| &= \left|\int_{\cA_t(x)\cap\overline B_{\frac1m}(a^m)} f(a)\lambda^m(\dif a) - f(a^0)\right| 
\\
&\le \sup_{a\in \cA_t(x)\cap\overline B_{\frac1m}(a^m)}|f(a)-f(a^0)| 
\\
&\le \sup_{a\in \cA_t(x)\cap\overline B_{\frac1m}(a^m)}|f(a)-f(a^m)| + |f(a^m)-f(a^0)|\; \xrightarrow[m\to\infty]{} \;0,
\end{align*}
i.e., $(\lambda^m)_{m\in\bN}$ converges to $\delta_{a^0}$ weakly. By \cref{lem:GvLSC} and \cref{lem:JstarInftyLSC}, we obtain
\begin{align*}
\min_{\lambda\in\varpi_t(x)}G^{\lambda}_t J^{*}_{t+1,T}(x) = \liminf_{m\to\infty} G^{\lambda^m}_t J^*_{t+1,T}(x) \ge G^{\delta_{a^0}}_t J^*_{t+1,T}(x),
\end{align*}
and thus $\delta_{a^0}\in D(x)$. Along with the fact that $(\delta_{a^\ell})_{\ell\in\bN}$ converges weakly if $(a^\ell)_{\ell\in\bN}\subset\bA$ converges, the closedness of $D(x)$ follows from \cref{lem:GvLSC} and \cref{lem:JstarInftyLSC} again. 
\end{proof}

\section{Examples}\label{sec:Examples}
In this section, we present several examples that illustrate the application of the theory developed in the previous sections. In \cref{subsec:ExampleSigma}, we discuss a few possible choices for $\sigma_{t,x}$ and verify the technical assumptions required for the DPP to be applicable. Subsequently, in \cref{subsec:FiniteExample}, we provide a finite horizon example focusing on optimal liquidation with limit order books. Finally, an infinite horizon example on autonomous driving is presented in \cref{subsec:InfiniteExample}. 

In general, risk-averse MDPs seldom admit closed-form solutions, and we typically rely on numerical methods to approximate the optimal policy. The numerical methods for solving risk-averse MDPs, particularly in a model-agnostic context, are currently under development. We refer to, e.g., \cite{tamar2015policy, Huang2017Risk, Kose2021Risk, coache2023reinforcement, coache2023conditionally},  and the reference therein for the recent progress. We  also provide some discussion on the relevant methods at the end of \cref{subsec:FiniteExample} and \cref{subsec:InfiniteExample}.

\subsection{Examples of $\sigma_{t,x}$}\label{subsec:ExampleSigma}
Below we present two examples of $\sigma_{t,x}$. The first one is essentially the distributional version of entropic risk measure. The second one draws inspiration from Kusuoka representation for law-invariant convex risk measure (cf. \cite[Theorem 2.1]{Jouini2006Law}).

\subsubsection{Entropic risk measure}
Consider $\tau:\bN\times\bX\to(0,\infty)$ and suppose that for $t\in\bN$, $\tau(t,\cdot)$ is continuous. For $(t,x)\in\bN\times\bX$, we define
\begin{align}\label{eq:DefEntropic}
\sigma^E_{t,x}(\mu) := \frac{1}{\tau(t,x)} \log\int_{\bR} \exp\big(\tau(t,x) y\big)
\;\mu(\dif y),\qquad \mu\in\cP_b(\bR).
\end{align}
This may be interpreted as an entropic risk measure where the risk-aversion parameter $\tau$ varies with time and state.
The verification that $\sigma^E_{t,x}$ with an arbitrarily fixed $(t,x)\in\bN\times\bX$ adheres to \cref{def:DCRM} can be established by drawing from the argument that demonstrates the conventional entropic risk measure as a convex risk measure (cf. \cite[Example 6.20]{Shapiro2021book}). To verify \cref{asmp:Main} (c), let $(x^n)_{n\in\bN}\subseteq\bX$ converge to $x^0\in\bX$ and $(\mu^n)_{n\in\bN}\subset\cP_b(K)$ converge to $\mu^0\in\cP_b(K)$, where $K\subset\bR$ is compact. Note that $\lim_{n\to\infty}\tau(t,x^n)^{-1} = \tau(t,x^0)^{-1}$. Moreover, since $(x,y)\mapsto\exp\big(-\tau(t,x)\, y\big)$ is continuous on $K$, by \cref{lem:ConvVaryingMeas} we have
\begin{align*}
\liminf_{n\to\infty} \int_{\bR} \exp\big(-\tau(t,x^n)\, y\big)\mu^n(\dif y) 
&= \liminf_{n\to\infty} \int_{K} \exp\big(-\tau(t,x^n) y\big)\mu^n(\dif y) \\
&\ge \int_{K} \exp\big(-\tau(t,x^0) y\big)\mu^0(\dif y) 
= \int_{\bR} \exp\big(-\tau(t,x^n) y\big)\mu^0(\dif y),
\end{align*}
By combining the above, we have verified \cref{asmp:Main} (c). Finally, we note that $\sigma^E_{t,x}$ is normalized.

\begin{remark}
It follows from the concavity of $\log$ that for any $\theta\in(0,1)$ and $\mu_1,\mu_2\in\cP(\bR)$, 
\begin{align*}
\sigma^E_{t,x}(\theta\mu_1+(1-\theta)\mu_2) \ge  \theta \sigma^E_{t,x}(\mu_1) + (1-\theta) \sigma^E_{t,x}(\mu_2).
\end{align*}
Considering \eqref{eq:DefG}, this inequality shows that utilizing entropic risk measures, $\sigma^E_{t,x}$, promotes randomized actions as optimizers. 
\end{remark}

\subsubsection{Kusuoka-type risk measure}\label{subsec:ExampleKusuoka}
Inspired by the Kusuoka representation of law invariant convex risk measure (cf. \cite[Theorem 7]{Frittelli2005Law}), we introduce the following Kusuoka-type convex risk measure at the level of distributions, where the set of spectral risk measures are allowed to vary with state. 

Let $\cP([0,1])$ be the set of probability measures on $\cB([0,1])$, endowed with the weak topology. For $t\in\bN$, we let $\cM_t:(\bX,\cB(\bX))\to 2^{\bM}$ be non-empty, closed-valued, and weakly measurable, that is, $\set{x\in\bX:\cM_t(x)\cap U\neq\emptyset}\in\cB(\bX)$ for any open $U\subseteq\bM$. Suppose $\cM_t$ is lower hemi-continuous\footnote{By \cite[Section 17.2, Lemma 17.5 and Section 18.1, Lemma 18.2]{Aliprantis2006book}), the lower hemi-continuity of $\cM_t$ implies weak measurability.}, that is, at any $x\in\bX$ for every open $U_\bM\subset\bM$ such that $U_\bM\cap\cM_t(x)\neq\emptyset$ there is an open $U_\bX\subseteq\bX$ such that $z\in U_\bX$ implies $U_\bM\cap\cM_t(z)\neq\emptyset$. Moreover, let $\beta_{t,x}:(\bM,\cB(\bM))\to\bR$ and assume that $(x,\eta)\mapsto\beta_{t,x}(\eta)$ is upper semi-continuous. For $(t,x)\in\bN\times\bX$, we define
\begin{align}\label{eq:DefKusuoka}
\sigma^K_{t,x}(\mu) := \sup_{\eta\in\cM_t(x)} \left\{ \int_{[0,1]} \overline\avar_{\kappa}(\mu)\,\eta(\dif\kappa) - \beta_{t,x}(\eta) \right\}, \quad \mu\in\cP_b(\bR).
\end{align} 
Note that $\sigma^K_{t,x}$ is normalized if and only if $\sup_{n\in\cM_t(x)}\set{-\beta(\eta)}=0$.  Using the properties of $\avar$ (cf. \cite[Section 6.2.4]{Shapiro2021book}), it is straightforward to check that, for $(t,x)\in\bN\times\bA$ arbitrarily fix, $\sigma_{t,x}$ satisfies \cref{def:DCRM}. We  use \cref{lem:SupIntLSC} to verify \cref{asmp:Main} (iii). To this end let $K\subset\bR$ be compact and consider 
\begin{align*}
f((x,\mu),\kappa) =  \overline\avar_{\kappa}(\mu), \quad \text{and} \quad \phi(x,\mu)=\cM_t(x) ,
\end{align*}
for $(x,\mu,\kappa)\in\bX\times\cP(K)\times[0,1]$. Note that $\phi$ is constant in $\mu$, thus lower hemi-continuous due to the definition of $\cM_t$. Since $f$ is constant in $x$, by \cref{rmk:DCRM}, $(x,\mu)\mapsto f((x,\mu),\kappa)$ is lower semi-continuous for any $\kappa\in[0,1]$. By \cref{lem:avarReg}, for any $\mu\in\cP_b(\bR)$, $\kappa\mapsto\avar_\kappa(\mu)$ is non-increasing and continuous on $[0,1]$, thus $\kappa\mapsto f((x,\mu),\kappa)$ is non-increasing and continuous on $[0,1]$ for any $(x,\mu)\in\bX\times\cP(K)$. By invoking \cref{lem:SupIntLSC}, we confirm that $\sigma^K_{t,x}$ defined in \eqref{eq:DefKusuoka} \cref{asmp:Main} (iii).

In view of \eqref{eq:DefKusuoka}, when $\cM_t(x)$ is a singleton, $\sigma^K_{t,x}$ represents the distributional version of spectral risk measure (cf. \cite{Acerbi2002Spectral}, \cite[Section 4.6]{Follmer2016book}). In general, \eqref{eq:DefKusuoka} permits $\cM_t(x)$ to extend beyond being merely a singleton. This flexibility enables us to incorporate a broader range of convex risk measures into a dynamic context.

\begin{remark}\label{rmk:CommentM}
In this remark, we provide some discussions on the impact of $\cM_t(x)$ on optimal actions.
\begin{itemize}
\item We claim that if $\cM_t$ is a singleton, then a deterministic Markovian action at time $t$ may achieve the optimal outcome. In view of \cref{prop:MarkovControlOpt}, the following observations are sufficient. Note that for any $\kappa\in(0,1]$, $\mu_1,\mu_2\in\cP_b(\bR)$ and $\theta\in[0,1]$,  
\begin{align*}
&\overline\avar_\kappa(\theta\mu_1+(1-\theta)\mu_2) 
\\
\quad&= \inf_{q\in\bR}\left\{  q + \kappa^{-1}\int_{\bX}\left(z-q\right)_+ (\theta\mu_1+(1-\theta)\mu_2)(\dif z) \right\}
\\
\quad&\ge \inf_{q\in\bR}\left\{  \theta q + \theta\kappa^{-1}\int_{\bX}\left(z-q\right)_+ \mu_1(\dif z) \right\} + \inf_{q\in\bR}\left\{  (1-\theta)q + (1-\theta)\kappa^{-1}\int_{\bX}\left(z-q\right)_+ \mu_2(\dif z) \right\} 
\\
\quad& = \theta\avar_\kappa(\mu_1) + (1-\theta)\avar_\kappa(\mu_2) \\
\quad &\ge \min\left\{\avar_\kappa(\mu_1),\avar_\kappa(\mu_2)\right\}.
\end{align*}
As for $\kappa=0$, note that $(\theta\mu_1+(1-\theta)\mu_2)((r,\infty))=0$ implies $\mu_i((r,\infty))=0$ for $i=1,2$, and thus
\begin{align*}
\overline\avar_0(\theta\mu_1+(1-\theta)\mu_2) \ge \min\left\{\avar_0(\mu_1), \avar_0(\mu_2)\right\}
\end{align*}
%Thus, if $\cM_t(x)$ is a singleton for all $x\in\bX$, then $\sigma^K_{t,x}$ satisfies the conditions of \cref{prop:MarkovControlOpt}. Consequently, deterministic actions can potentially attain the optimal.
\item 
%If $\cM_t$ is not a singleton, in analogous to \cref{app:Example}, there exists certain scenario where randomized action is needed to attain the optimal. The potential need for randomized action can be explained using the following two-play game heuristics. Recall the definitions of $\wt P^v_{t,x,\lambda}$ from \eqref{eq:DefPtilde}. When considering $\inf_{\lambda}\sigma^{K}_{t,x}\left(\wt P^v_{t,x,\lambda}\right)$, we may view $\sup_{\eta\in\cM_t(x)}$ as an adversarial agent who control $\eta$. However, but the adversarial agent has to make a decision without the knowledge of $\lambda$. For the original agent who controls $\lambda$, it is natural to use a randomized action to address the adversary under suitable scenario. Finally, we note that, unlike entropic risk measure, Kusuoka-type risk measure do not always encourages randomized action even if $\cM_t$ is not a singleton. How different choices of $\cM_t$ affects the optimal action appears to be quite elusive, and will be investigated elsewhere.
If $\cM_t$ is not a singleton, analogous to the scenario in Appendix \ref{app:Example}, there exists certain circumstances where randomized actions are necessary to achieve optimality. The potential need for randomized action can be elucidated using the following two-player game heuristic. Recall the definitions of $\wt P^v_{t,x,\lambda}$ from equation \eqref{eq:DefPtilde}. When considering $\inf_{\lambda\in\varpi_t(x)}\sigma^{K}_{t,x}\left(\wt P^v_{t,x,\lambda}\right)$, we can interpret $\sup_{\eta\in\cM_t(x)}$ as an adversarial agent who controls $\eta$. According to the order of $\inf$ and $\sup$, this adversarial agent has access to $\lambda$. However, she must pick $\eta$ without the knowledge of the realized action. For the original agent who controls $\lambda$, under certain circumstances, it is reasonable to employ randomized action to counteract the adversary. Lastly, it's important to note that, unlike the entropic risk measure, the Kusuoka-type risk measure does not always encourage randomized action, even if $\cM_t$ is not a singleton. The way different choices of $\cM_t$ affect the optimal action seems to be quite complex and will be explored in future work.
\end{itemize}
\end{remark}

\subsection{A finite horizon example: optimal liquidation}\label{subsec:FiniteExample}

In this section, we showcase an example that demonstrates optimal liquidation with limit order book utilizing the proposed risk-averse MDP. We would draw attention to several works that offer different perspectives on trading with limit order books \cite{gueant2013dealing, cartea2014buy, Bayraktar2014Liquidation,cartea2015risk, cartea2015optimal, Jacquier2018Optimal, Fouque2022Optimal, Nadtochiy2023Consistency}. We also refer to \cite{Cartea2015book} for an overview of the clearing mechanism of limit order books. As all stock market activities occur discretely (e.g., tick size, integer trading volume), we construct the example with discrete state and action spaces. With discrete state and action spaces, \cref{asmp:Main} is automatically satisfied. On the other hand, taking into account the inherent numeric structure of stock prices and other factors present in the stock market, formulating the problem into an MDP with continuous state and action spaces could provide sensible approximations. Doing so, however,  demands more technicalities and will be studied elsewhere.

Consider a single stock market. We assume zero interest rate and we accordingly set $\gamma=1$. Suppose an agent wishes to liquidate $u_0\in\bN$ amount of shares by the terminal time $T\in\bN$. For simplicity, let $\bS=\set{0,1,\dots,N_S}\subset \bN_0$ be the set of possible values of the price of a stock. %Let $S=(S_t)_{t\in\bN}$ be a $\bF$-adapted $\bS$-valued process that represent the movement of the stock price. 
At each time $t\in\set{1,\dots,T}$, potential sellers and buyers post their intended trading prices and volumes on the market. This  results in a limit order book, $m=(m^0,m^1,\dots,m^{N_S})$, with each entry being an integer reflecting the aggregated volume submitted for trading at the corresponding price. Positive integers indicate a collective intention to sell while negative integers indicate a collective intention to buy. For example, $m^1=10$ signifies a collective intention to sell $10$ units at price $1$; $m^0=-3$ signifies a collective intention to buy $3$ units at price $0$. A valid limit order book $m$ must be such that there is a $\check i\in\set{0,0.5,1.5,\dots,N_S-0.5,N_S}$ with $m_j \le 0$ for all $j\in\set{0,1,\dots,N_S}$ and $j<\check i$ and $m_j\ge 0$ for all $j\in\set{0,1,\dots,N_S}$ with $j>\check i$. It is reasonable to assume that $m^0=-\infty$. We thus let $\bM\subset (\{-\infty\}\cup\bZ)^{N_S+1}$ be the set of valid limit order books with $m^0=-\infty$ -- i.e., that there is an infinite volume for purchasing the asset at price $0$. The combination of the prices and allowed limit order books constitutes the state space $\bX=\bM\times\bN$. 

Regarding the potential actions, the agent may post her intended volume for selling at each price, i.e., $a=(a^0,a^1,\dots,a^{N_S})$. As the agent aims to liquidate their shares, buying and short selling are not allowed. We set $\bA=\bZ_{\ge 0}^{N_S+1}$.  At $t=1,\dots,T-1$, suppose the remaining units of stock is $u$, the admissible domain of actions is $\cA_t(u)=\big\{a\in\bZ_{\ge 0}: \sum_{s=0}^{N_S} a^s\le u\big\}$, i.e., they can post limit orders to sell in total at most the amount of shares remaining. At $t=T$, the agent must liquidate all $u$ units of remaining stock, resulting in the singleton admissible domain of actions $\cA_T(u)=\set{(u,0,\dots,0)}$. As the state space is $\bX=\bM\times\bN$, we can extend $\cA_t$ to accept the pair $(m,u)$ as input but be constant in $m$.

Suppose the market generates the next limit order book, $m_{t+1}$, given the current state $(m_t,u_t)\in\bX$ and (realized) action $a_t\in\bA$ in a Markovian fashion prescribed by some probability on $\bM$, $B\mapsto\mathring P(t,(m_t,u_t),a_t,B)$ for $B\subseteq{\bM}$. With the risk induced by time-discretization and processing delay in mind, we additionally assume that the clearing occurs between $t$ and $t+1$ in the following way:
\begin{itemize}
\item Let $\overline m_t = \left\lceil\frac12(m_t+m_{t+1})\right\rceil$, i.e.,  $\overline m_t^s = \left\lceil\frac12(m^s_t+m^s_{t+1})\right\rceil$ for $s=0,\dots,N_S$.
\item The remaining units to be liquidated at $t+1$ is $u_{t+1} = u_t  -\Delta(a_t,\overline m_t)$, where
\begin{align*}
\Delta(a,\overline m)=\sum_{s=0}^{N_S} \1_{(-\infty,0)}(\overline m^{s})\left((-\overline m^{s}) \wedge \left(-\sum_{r=0}^{s} a^r_t - \sum_{r=s+1}^{N_s} (\overline m^{r})_{-}\right)_+\right).
\end{align*}
\item The cash gained from action $a_t$ is $g(a_t,\overline m_t)$, where
\begin{align*}
g(a,\overline m) =  \sum_{s=0}^{N_S}s \1_{(-\infty,0)}(\overline m^{s})\left((-\overline m^{s}) \wedge \left(-\sum_{r=0}^{s} a^r_t - \sum_{r=s+1}^{N_s} (\overline m^{r})_{-}\right)_+\right).
\end{align*}
\end{itemize}  
It is important to clarify that the clearing rules described above are stylized as, in the current model, there are no time stamps for limit order arrivals/cancellations. More detailed models can be formulated by making various assumptions on the timing of the orders. Following the set up above, we obtain  the transition kernel
\begin{align*}
P(t,(m,u),a,D) = \int_{\bM}\1_{D}\left(m,u-\Delta\left(a,\frac12(m+y)\right)\right) \mathring P(t,(m,u),a,\dif y),\quad D\subseteq \bV\times\bN.
\end{align*}
and stationary latent cost 
\begin{align*}
C((m,u),a,(m',u')) = -g\left(a,\frac12(m+m')\right) .
\end{align*}
At time $t=0,1,\dots,T-1$, suppose we are given a value function $v$ for $t+1$. Let $\lambda$ be a probability measure on $\bA$ representing the randomized action. The above leads to the following regular conditional distribution 
\begin{multline*}
\wt P^{v}_{t,(m,u),\lambda}(B)\\
= \int_{\bA}\int_{\bM}C\left((m,u),a,\left(m',u-\Delta\left(a,\frac12(m+m')\right)\right)\right) + v\left(m',u-\Delta\left(a,\frac12(m+m')\right)\right) \\
\mathring P(t,(m,u),a,\dif m') \lambda(\dif a).
\end{multline*}

For the choice of $\sigma_{t,x}$, we may choose, e.g., the weighted average of entropic risk measure $\sigma^E_{t,x}$ and Kusuoka-type risk measure $\sigma^K_{t,x}$. The weight may vary across time and state depending on the agents specific requirements. For example, a simple time varying choice could be 
\begin{align*}
\sigma_{t,(m,u)}(\mu) = \left(1-\frac{t}{T}\right)\sigma^E_{t,(m,u)}(\mu) + \frac{t}{T}\sigma^K_{t,(m,u)}(\mu),
\end{align*}
where $\sigma^E_{t,(m,u)}$ is defined in \eqref{eq:DefEntropic} with $\tau\equiv 1$, and  $\sigma^K_{t,(m,u)}$ is defined \eqref{eq:DefKusuoka} with $\beta_{t,x}\equiv0$ for any $(t,x)$ and
\begin{align*}
\cM_t(m,u) = \cM_t(u) = \left\{\frac12\delta_1 + \frac12\delta_{0.5\frac{u}{u_0}},\; \delta_{\frac{u}{u_0}}\right\}. 
\end{align*}

Regarding the usage of $\sigma^E_{t,(m,u)}$, we underscore one particular viewpoint: in real-world scenarios, other market participants may infer and take advantage of an agent's intention to liquidate. While modeling this dynamic is complex, an alternative approach involves using the entropic risk measure to promote randomization, which could ultimately help reduce information leakage. As for $\sigma^K_{t,(m,u)}$, it encourages risk aversion when a substantial number of units remain, and gradually transits toward profit-seeking as the liquidation process progresses. 

Finally, the optimal liquidation problem may be solved in a risk averse manner by using \cref{thm:FiniteDPP}. Let $O\equiv 0$. We start from $t=T$ with
\begin{align*}
J^*_{T,T}(m,u) = \sigma_{T,(m,u)}\left(\wt P^{O}_{T,(m,u),\delta_{(u,0,\dots,0)}}\right).
\end{align*}
For $t=1,\dots,T-1$, we find
\begin{align*}
J^*_{t,T}(m,u) = \inf_{\lambda\in\cP(A_t(u))} \sigma_{t,(m,u)}\left(\wt P^{J^*_{t+1,T}}_{t,(m,u),\lambda}\right) \quad\text{and}\quad \pi_t(x)\in\argmin_{\lambda\in\cP(A_t(u))} \sigma_{t,(m,u)}\left(\wt P^{J^*_{t+1,T}}_{t,(m,u),\lambda}\right),\quad x\in\bX.
\end{align*}

Regarding numerical methods for solving this example, it is noteworthy to consider the max-min structure indicated in \eqref{eq:DefAVaR} and \eqref{eq:DefKusuoka}. A promising initial approach could be the min-max $Q$-learning as proposed in \cite{Huang2017Risk}. Additionally, for the effective integration of randomized actions, the distributional method in \cite{Cheng2023Distributional} may offer further insights and efficiency.

\subsection{An infinite horizon example: auto-driving}\label{subsec:InfiniteExample}
In this example, we consider an autonomous driving scenario. For a comprehensive understanding on the role of risk aversion in autonomous robots, we direct the reader to \cite{Majumdar2019How} and \cite{Wang2022Risk}. We refer to \cite{Yurtsever2020Survey} for an overview of autonomous driving in real-world settings. In the subsequent discussion, we  assume a perpetual driving experiment within a 2D environment.

We start by constructing the state space, with the vehicle's surroundings as a crucial component. The surroundings are modeled by a positive measure on the rectangle $R = [-W, W] \times [-L, L]$, with $W, L \ge 0$ representing the chosen width and length of the surrounding box. Coordinate $(0,0)$ corresponds to the center of the agent's vehicle. We let $k \ge 0$ and define $\cM_k(R)$ as the set of positive measures on $(R, \cB(R))$ with $m(R) \le k$ for any $m \in \cM_k(R)$. We endow $\cM_k(R)$ with the weak topology. A positive measure assigned to an area indicates obstacles in that area, such as other vehicles, boulders, or animals. In principle,  the 2D shape of obstacles may be encoded into $m$ to increase the realism of the model.

In addition to the surroundings, we may include other factors describing the vehicle's operating conditions into the state space. For simplicity, however, we consider only the 2D velocity vector and the deviation from the center of the road, assuming that the road is straight and has a constant width $W_0$. We use $\bU\subset\bR^2$ for the domain of the vehicle's velocity and $\bD=[-W_0,W_0]$ for the domain of the vehicle's deviation from the center of the road.

To sum up, the state space is $\bX = \cM_k(R) \times \bU \times \bD$. We denote $x = (m, u, d) \in \bX$.

We set $\bA = \bR^2$ to reflect acceleration of the vehicle. The admissible action domain $\cA(u)$ is stationary, depends only on the vehicle's velocity, and must be modeled in a way that satisfies \cref{asmp:Main} (ii). For example, if $\cA(u)=\set{g_\alpha(u):\alpha\in\set{1,\dots,k}}$ and $u\mapsto g_\alpha(u)$ is continuous and uniformly bounded for all $\alpha$, then \cref{asmp:Main} (ii) is satisfied. Furthermore, $\cA(u)$ should respect road regulations. The detailed construction of $\cA(u)$ is omitted.

We define $P(x, a, \cdot)$ as a time-homogeneous probability measure on $(\bX, \cB(\bX))$ satisfying \cref{asmp:Main} (i) and governing the state transition. In view of the physical nature embedded in the problem, we impose the following constraint on the transition kernel 
$$P\left((m,u,d),a,\cM_b(R)\times\{u+a\Delta t\}\times\left\{d+u_1\Delta t+\frac12 a^2_1\Delta t^2\right\}\right)=1,$$ 
where $\Delta t$ is a parameter reflecting the resolution of time discretization. This constraint imposes a deterministic transition mechanism in the velocity and position of the vehicle, as the agent controls their acceleration and the physical laws of motion must be obeyed. This excludes the use of a strongly continuous transition kernel\footnote{In the current context, a strongly continuous transition kernel means a kernel is set-wise continuous. For example, we consider $Q:\bX\to\cP(\bR)$ strongly continuous if $x\mapsto [Q(x)](A)$ is continuous for any $A\in\cB(\bR)$.}, but is compatible with the weak continuity assumed in \cref{asmp:Main} (i). 

We next discuss the choice of cost and $\sigma_{t,x}$. In the perceptual driving experiment, it is appropriate to adopt time-homogeneous cost functions and $\sigma_{t,x}$. To account for time discretization and potential processing delays, we consider a latent cost function, such as 
\begin{align*}
C(x,a,x') = C(x') = \int_{R} g(r) m'(\dif r), \quad g(r) = \kcol\,\1_{(-W_v,W_v)\times(-L_v,L_v)}(r) + \exp(-\kcaution \,r),
\end{align*}
where $\kcol,\kcaution\ge 0$ are chosen parameters corresponding to a collision and caution, and $W_v,L_v\ge 0$ reflect the size of the agent's vehicle. Despite the indicator $\1_{(-W_v,W_v)\times(-L_v,L_v)}$,  the cost is lower semi-continuous as required in \cref{asmp:Main} (iv). Regarding $\sigma_x$, to encourage predictable driving behavior, in view of \cref{rmk:CommentM}, we consider the Kusuoka-type risk measure in \eqref{eq:DefKusuoka} with singleton $\cM_t$. More precisely, we let $\sigma_x$ take the form of a spectral risk measure 
\begin{align*}
\sigma_{m,u,d}(\mu) = \int_{[0,1]}\avar_{\kappa}(\mu)\; \eta_{m,u,d}(\dif \kappa).
\end{align*}
Following the discussion in \cref{subsec:ExampleKusuoka}, if $(m,u,d)\mapsto\eta_{m,u,d}$ is weakly continuous, then \cref{asmp:Main} (iii) is satisfied. We can choose $\eta_{m,u,d}$ in a principled way, for example, by assigning more weight to small $\kappa$ values when $|u|$ is large.

To formulate the infinite horizon DPP via \cref{thm:InfiniteDPP}, we select a discount factor $\gamma \in (0, 1)$. For $(m,u,d)\in\bX$, $\lambda\in\cP(\bA)$, and $v\in\ell^\infty(\bX,\cB(\bX))$ and 
\begin{align*}
\wt P^v_{(m,u,d),\lambda}(B) = \int_{\bA} \int_{\cM_k(R)}\1_{B}\left(\int_R g(r)m'(\dif r)+\gamma v(m',u',d')\right) P((m,u,d),a,\dif m')\lambda(\dif a),\quad B\in\cB(\bR).
\end{align*}
Here, $\wt P^v_{(m,u,d),\lambda}$ is time-homogeneous as the transition kernel and cost does not dependent on time. We solve the fixed point equation with unknown $v\in\ell^\infty(\bX,\cB(\bX))$:
\begin{align*}
v(m,u,d) = \inf_{\lambda\in\cP(\cA(m,u,d))}\int_{[0,1]}\overline\avar_{\kappa}\left(\wt P^v_{(m,u,d),\lambda}\right) \eta_{m,u,d}(\dif \kappa), \quad (m,u,d)\in\bX
\end{align*}
for the optimal value function $v^*$, and find the corresponding optimal policy
\begin{align*}
\pi(m,u,d) \in \argmin_{\lambda\in\cP(\cA(m,u,d))}\int_{[0,1]}\overline\avar_{\kappa}\left(\wt P^v_{(m,u,d),\lambda}\right) \eta_{m,u,d}(\dif \kappa), \quad (m,u,d)\in\bX.
\end{align*}
In view of \cref{rmk:Stationary}, the stationary setting above indeed allows certain stationary policies to attain the optimal.

Regarding numerical methods for solving this example, given that here the state space is continuous, approximation techniques tailored for risk-averse MDPs, e.g. \cite{Yu2018Approximate}, could be beneficial. We can also leverage the  elicitability of spectral risk measures and employ techniques developed in \cite{coache2023conditionally}. It is important to note that most existing methods are designed for discrete spaces or $\bR^d$, and adapting them to handle $\cP(\mathbb{R}^d)$ necessitates additional modification. A promising initial reference for this adaptation is \cite{Peyre2019book}.

\appendix

\small 

\section{Example illustrating the necessity of randomized actions}\label{app:Example}
We let
\begin{align*}
\rho(Z) = \max\left\{0.8\bE(Z)+0.2\cvar_{0.1}(Z),\,\cvar_{0.5}(Z)\right\},
\end{align*}
where $\cvar_\kappa(Z):=\inf_{q\in\bR}\set{q+\kappa^{-1}\bE(Z-q)_+}$ for $\kappa\in(0,1]$ and $(x)_+:=\max\set{x,0}$.
Let us consider controls 0 and 1 with the following random costs,
\begin{align*}
C^0 = \begin{cases}
0,&p=0.9,\\
5,&p=0.1,
\end{cases}
\quad\text{and}\quad C^1 = \begin{cases}
0,&p=0.5,\\
1.4,&p=0.5.
\end{cases}
\end{align*}
Let $\theta\in[0,1]$. We also consider randomly and independently choosing between control 0 and 1 with $1-\theta$ and $\theta$ probability, respectively. Accordingly, we define
\begin{align*}
C^\theta := I^\theta C^0 + (1-I^\theta) C^1,\quad\theta\in[0,1],
\end{align*}
where $I^\theta$ is a random variable independent of $C^0$ and $C^1$, and $\bP(I^\theta=1)=1-\theta$, $\bP(I^\theta=0)=\theta$. It follows that for $\theta\in[0,1]$, $C^\theta$ has the distribution below,
\begin{align*}
C^\theta = \begin{cases}
0,& p = 0.9 - 0.4\theta,\\
1.4,& p = 0.5\theta,\\
5,& p = 0.1-0.1\theta.
\end{cases}
\end{align*}
It can be shown that $\rho(C^0)=\rho(C^1)=1.4>\rho(C^{0.5})=1.2$.

\section{Technical Lemmas}\label{app:Lemmas}

\begin{lemma}\label{lem:avarReg}
Let $Z\in L^\infty(\Omega,\sH,\bP)$, $\avar_\kappa(Z):=\inf_{q\in\bR}\set{q+\kappa^{-1}\bE(Z-q)_+}$ for $\kappa\in(0,1]$ and\\ $\avar_0(Z):=\|Z\|_\infty$. Then, $\overline\avar_\kappa(Z)$ is non-increasing and continuous in $\kappa\in[0,1]$.
\end{lemma}
\begin{proof}
See \cite[Section 6.2.4, Remark 22]{Shapiro2021book}.
\end{proof}

\begin{lemma}\label{lem:asconvMeasurable}
Let $(\Omega,\sA,\bP)$ be a complete probability space. Let $(Z_n)_{n\in\bN}$ be a sequence of real-valued $\sA$-$\cB(\bR)$ random variable converging to $Z$ $\bP$-almost surely. Then, $Z$ is $\sA$-measurable. 
\end{lemma}
\begin{proof}
This is an immediate consequence of the fact that pointwise convergence preserves measurability \cite[Section 4.6, Lemma 4.29]{Aliprantis2006book} and that $\sA$ contains all $\bP$-negligible sets.
\end{proof}

\begin{lemma}\label{lem:IntfMeasurability}
Let $(\bX,\sX)$ and $(\bY,\sY)$ be measurable spaces. Consider nonnegative $f:(\bX\times\bY,\sX\otimes\sY)\to(\bR,\cB(\bR))$ and $M:(\bX,\sX)\to(\cP,\cE(\cP))$, where $\cP$ is the set of probability measure on $\sY$ and $\cE(\cP)$ is the corresponding evaluation $\sigma$-algebra. Then, $x\mapsto\int_\bY f(x,y)\,[M(x)](\dif y)$ is $\sX$-$\cB(\bR)$ measurable. 
\end{lemma}
\begin{proof}
We first consider $f(x,y)=\1_D(x,y)$, where $D\in\sX\otimes\sY$. Let $\sD$ consist of $D\in\sX\otimes\sY$ such that $x\mapsto\int_\bY f(x,y) [M(x)](\dif y)$ is $\sX$-$\cB(\bR)$ measurable. Note $\set{A\times B:A\in\sX,B\in\sY}\subseteq\sD$ because $\int_\bY \1_{A\times B}(x,y) [M(x)](\dif y) = \1_{A}(x) \, [M(x)](B)$ and 
\begin{align*}
\set{x\in\bX:[M(x)](B)\in C} = \set{x\in\bX:M(x)\in\set{\zeta\in\cP:\zeta(B)\in C}}\in\sX,\quad C\in\cB([0,1]).
\end{align*}
If $D^1,D^2\in\sD$ and $D^1\subseteq D^2$, then
\begin{align*}
\int_{\bY}\1_{D^2\setminus D^1}(x,y)\,[M(x)](\dif y) = \int_{\bY}\1_{D^2}(x,y)\,[M(x)](\dif y) - \int_{\bY}\1_{D^1}(x,y)\,[M(x)](\dif y)
\end{align*}
is also $\sX$-$\cB(\bR)$ measurable. Similarly, if $D^1,D^2\in\sD$ are disjoint, then $D^1\cup D^2\in\sD$. If $(D^n)_{n\in\bN}\subseteq\sD$ is increasing and $D^0=\bigcup_{n\in\bN}D^0$, then by monotone convergence, 
\begin{align*}
\lim_{n\to\infty}\int_{\bY}\1_{D^n}(x,y)\,[M(x)](\dif y) = \int_{\bY}\1_{D^0}(x,y)\,[M(x)](\dif y),\quad x\in\bX.
\end{align*}
It follows from \cite[Section 4.6, Lemma 4.29]{Aliprantis2006book} that $D^0\in\sD$. Invoking monotone class theorem (cf. \cite[Theorem 1.9.3 (ii)]{Bogachev2006book}), we yield $\sD=\sX\otimes\sY$.

Now let $f$ be any non-negative measurable function. Note that $f$ can be approximated by a sequence of simple function $(f^n)_{n\in\bN}$ such that $f^n\uparrow f$ (cf. \cite[Section 4.7, Theorem 4.36]{Aliprantis2006book}). Since $\sD=\sX\otimes\sY$, for $n\in\bN$, $x\mapsto\int_\bY f^n(x,y) [M(x)](\dif y)$ is $\sX$-$\cB(\bR)$ measurable. Finally, by monotone convergence and \cite[Section 4.6, Lemma 4.29]{Aliprantis2006book} again, we conclude the proof.
\end{proof}

In what follows, we let $\bY$ and $\bZ$ be two topological space.  
Below is a reduction of Portmanteau theorem. %(cf. \cite[Corollary 8.2.4 and Corollary 8.2.5]{Bogachev2007book}).
\begin{lemma}\label{lem:Portmanteau}
Suppose $\bY$ is a metric space endowed with Borel $\sigma$-algebra $\cB(\bY)$. Let $(\mu^n)_{n\in\bN}$ be a sequence of probability measures on $\cB(\bY)$, and $\mu$ a probability measure on $\cB(\bY)$. Then the following conditions are equivalent:
\begin{enumerate}
\item[(a)] $(\mu_n)_{n\in\bN}$ converges weakly to $\mu$;
\item[(b)] $\limsup_{n\to\infty}\mu^n(F)\le\mu(F)$ for any closed set $F\subset\bY$;
\item[(c)] $\liminf_{n\to\infty}\mu^n(U)\ge\mu(U)$ for any open set $U\subset\bY$;
\item[(d)] $\limsup_{n\to\infty}\int_\bY f(y)\mu^n(\dif y) \le \int_\bY f(y)\,\mu(\dif y)$ for any upper semicontinuous $f\in\ell^\infty(\bY,\cB(\bY))$;
\item[(e)] $\liminf_{n\to\infty}\int_\bY f(y)\mu^n(\dif y) \ge \int_\bY f(y)\,\mu(\dif y)$ for any lower semicontinuous $f\in\ell^\infty(\bY,\cB(\bY))$;
\item[(f)] $\lim_{n\to\infty}\mu_n(A)=\mu(A)$ for any $A$ with $\mu(\partial A)=0$, where $\partial A$ is the boundary of $A$.
\end{enumerate}
\end{lemma}
\begin{proof}
See \cite[Section 15.1, Theorem 15.3 and Theorem 15.5]{Aliprantis2006book}.
\end{proof}

Let $\cB(\bY)$ be the Borel $\sigma$-algebra on $\bY$, and $\cP$ be the set of probability measures on $\cB(\bY)$. We endow $\cP$ with weak topology and let $\cB(\cP)$ be the corresponding Borel $\sigma$-algebra. Let $\cE(\cP)$ be the $\sigma$-algebra on $\cP$ generated by sets $\set{\zeta\in\cP:\zeta(A)\in B},\,A\in\cB(\bY),\,B\in\cB([0,1])$. Equivalently, $\cE(\cP)$ is the $\sigma$-algebra generated by sets $\set{\zeta\in\cP:\int_{\bY}f(y)\zeta(\dif y) \in B}\,$ for any real-valued bounded $\cB(\bY)$-$\cB(\bR)$ measurable $f$ and $B\in\cB(\bR)$. The lemma below states that $\cB(\cP)$ and $\cE(\cP)$ are equivalent.
\begin{lemma}\label{lem:sigmaAlgBE}
If $\bY$ is a separable metric space, then $\cB(\cP)=\cE(\cP)$.
\end{lemma}
\begin{proof}
Notice that the weak topology of $\cP$ is generated by sets $\set{\zeta\in\cP:\int_{\bY}f(y)\zeta(\dif y)\in U}$ for any $f\in C_b(\bY)$ and open $U\subseteq\bR$. Therefore, $\cB(\cP)\subseteq\cE(\cP)$. On the other hand, by \cite[Section 15.3, Theorem 15.13]{Aliprantis2006book}, for any bounded real-valued  $\cB(\bY)$-$\cB(\bR)$ measurable $f$ and any $B\in\cB(\bR)$, we have $\set{\zeta\in\cP:\int_{\bY}f(y)\zeta(\dif y) \in B}\in\cB(\cP)$. The proof is complete.
\end{proof}

\begin{lemma}\label{lem:ConvVaryingMeas}
Let $\bY\times\bZ$ be a separable metric space endowed with product Borel $\sigma$-algebra $\cB(\bY)\otimes\cB(\bZ)$. Let $f\in\ell^\infty(\bY\times\bZ,\cB(\bY)\otimes\cB(\bZ))$ be lower semi-continuous. Let $(y^n)_{n\in\bN}\subset\bY$ converges to $y^0\in\bY$ and let $(\mu^n)_{n\in\bN}$ be a sequence of probability measure on $\cB(\bZ)$ converging weakly to $\mu^0$. Then,
\begin{align*}
\liminf_{n\to\infty}\int_\bZ f(y^n,z)\, \mu^n(\dif z) \ge \int_\bZ f(y^0,z)\, \mu^0(\dif z).
\end{align*}
\end{lemma}
\begin{proof}
To start with, let us $\delta_y(B) := \1_{B}(y)$, i.e., $\delta_y$ is the dirac measure on $y$. It follows that $(\delta_{y^n})_{n\in\bN}$ converges weakly to $\delta_{y^0}$. Therefore, by \cite[Theorem 2.8 (ii)]{Billingsley1999book}, $(\delta_{y^n}\otimes\mu^n)_{n\to\bN}$ converges to $\delta_{y^0}\otimes\mu^0$. Due to Lemma \ref{lem:Portmanteau} (e), we yield
\begin{align*}
\liminf_{n\to\infty}\int_\bZ f(y^n,z)\, \mu^n(\dif z) &= \liminf_{n\to\infty}\int_{\bY\times\bZ} f(y,z)\, \delta_{y^n}\otimes\mu^n(\dif y\,\dif z)\\
&\ge \int_{\bY\times\bZ} f(y,z)\, \delta_{y^0}\otimes\mu^0(\dif y\,\dif z) = \int_\bZ f(y^0,z)\, \mu^0(\dif z),
\end{align*}
where we have used Fubini's theorem in the first and last equality.
\end{proof}

\begin{lemma}\label{lem:InffLSC}
Let $f$ be a real-valued lower semi-continuous function on $\bY\times\bZ$. Let $\phi:\bY\to 2^\bZ$ be nonempty compact valued and upper hemi-continuous.  Then, $\inf_{z\in\phi(y)}f(y,z)$ is attainable for all $y\in\bY$ and $\inf_{z\in\phi(y)}f(\cdot,z)$ is lower semi-continuous on $\bY$. 
\end{lemma}
\begin{proof}
It follows from the compactness of  $\phi(y)$ and the lower semi-continuity of $f$ implies that $\inf_{z\in\phi(y)}f(y,z)$ is attainable for $y\in\bY$. By (cf. \cite[Section 17.5, Lemma 17.30]{Aliprantis2006book}), $y\mapsto\sup_{z\in\phi(y)}(-f(y,z))$ is upper semi-continuous, and thus $y\mapsto\inf_{z\in\phi(y)}f(y,z)$ is lower semi-continuous.
\end{proof}

\begin{lemma}\label{lem:SupIntLSC}
Let $\bY$ be a seperable metric space, $\phi:\bY\to 2^\bM$ be lower hemi-continuous, $f\in\ell^\infty(\bY\times[0,1],\sY\otimes\cB([0,1]))$, and $\beta:\bY\times\cM\to\bR$ be upper semi-continuous. Suppose $f(\cdot,\kappa)$ is lower semi-continuous for $\kappa\in[0,1]$ and $f(y,\cdot)$ is continuous and non-increasing for $y\in\bY$.  Then, $$y\mapsto\sup_{\eta\in\phi(y)}\left\{\int_{[0,1]}f(y,\kappa)\eta(\dif\kappa)-\beta(y,\eta)\right\}$$ is lower semi-continuous on $\bY$.
\end{lemma}
\begin{proof}
We first show that $f$ is lower semi-continuous. To this end let $(y^n)\subseteq\bY$ and $(\kappa^n)\subseteq[0,1]$ converge to $y^0\in\bY$ and $\kappa^0$, respectively. If $\kappa^0=1$, then $\liminf_{n\to\infty} f(y^n,\kappa^n) \ge \liminf_{n\to\infty}f(y^n,1) \ge f(y^0,\delta)$. If $\kappa^0<1$, then for any $\delta\in(\kappa^0,1]$ we have $\liminf_{n\to\infty} f(y^n,\kappa^n) \ge \liminf_{n\to\infty}f(y^n,\delta) \ge f(y^0,\delta)$. Letting $\delta\downarrow\kappa^0$, by the continuity of $f(y,\cdot)$, we yield $\liminf_{n\to\infty} f(y^n,\kappa^n)\ge f(y^0,\kappa^0)$. 

In view of Lemma \ref{lem:ConvVaryingMeas}, $(y,\eta)\mapsto\int_{[0,1]}f(y,\kappa)\eta(\dif\kappa)$ is lower semi-continuous. It follows that $(y,\eta)\mapsto\int_{[0,1]}f(y,\kappa)\eta(\dif\kappa)-\beta(y,\eta)$ is also lower semi-continuous. Finally, in view of \cite[Section 17.5, Lemma 17.29]{Aliprantis2006book}, the proof is complete.
\end{proof}

\newpage

\begin{minipage}{0.92\linewidth}
\section{Glossary of notations}\label{app:notations}

\begin{center}
\resizebox{\columnwidth}{!}{
\begin{tabular}[ht]{c c} 
\hline
\bf{Notations} & \makecell{\bf{Definitions}}\\
\hline
$\bN, \bN_0$ & \makecell{\small $\set{1,2,\dots}$ and $\set{0,1,2,\dots}$. \small See \cref{sec:Setup}.}\\
$\ell^\infty(\bY,\sY)$ & \makecell{\small Set of bounded $\sY$-$\cB(\bR)$ measurable functions. \small See \cref{sec:Setup}.}\\
$\delta_y$ & \makecell{\small Dirac measure at $y$. \small See \cref{sec:Setup}.}\\
$(\Omega,\sH,\bP),\bF,\bG,\bU,\sF_t,\sG_t,\sU_t$ & \makecell{\small Probability space, filtrations and $\sigma$-algebras. \small See \cref{sec:Setup}.}\\
$\bX,\bA$ & \makecell{\small State and action spaces. See \cref{sec:Setup}. }\\ 
$\Xi,\Lambda$ & \makecell{\small Set of probability measures on $\cB(\bX)$ and $\cB(\bA)$. See \cref{sec:Setup}. }\\ 
$\cB(\Xi),\cB(\Lambda),\cE(\Xi),\cE(\Lambda)$ & \makecell{\small Borel $\sigma$-algebras and evaluation $\sigma$-algebras of $\Xi$ and $\Lambda$. See \cref{sec:Setup}. }\\ 
$\cA_t$ & \makecell{\small Admissible action domain. See \cref{sec:Setup}. }\\ 
$\varpi_t$ & \makecell{\small Set of probability measures with support contained by $\cA_t$. See \cref{sec:Setup}. }\\ 
$\Pi_t$ & \makecell{\small Set of $\pi_t:(\bX,\cB(\bX))\to(\bA,\cE(\bA))$ such that $\pi_t(x)\in\varpi_t(x)$. See \cref{sec:Setup}. }\\ 
$\fp,\Pi$ & \makecell{\small Abbreviation of Markovian policy $(\pi_t)_{t\in\bN}$ and the set of $\fp$. See \cref{sec:Setup}. }\\ 
$\cP(\bR^d),\cP_b(\bR^d)$ & \makecell{\small Set of probabilities on $\cB(\bR^d)$ and set of probabilities with bounded support.\\ See \cref{sec:Setup}. }\\ 
$P^{Y|\sG}$ & \makecell{\small Regular conditional distribution $Y$ given $\sG$. See \cref{def:RegCondDist}. }\\ 
$X_t,A_t,\fX,\fA$ & \makecell{\small State and action process, abbreviations of $(X_t)_{t\in\bN}$ and $(A_t)_{t\in\bN}$.\\ \small See \cref{subsec:controlledP}.}\\
$\sigma$ & \makecell{\small Convex risk measure defined at the level of distributions.}  See \cref{def:DCRM}.\\
$\overline\avar_\kappa$ & \makecell{\small Average valued at risk. See \eqref{eq:DefAVaR} and \eqref{eq:DefAVAR0}.}\\
$\sigma_{t,x}$ & \makecell{\small Time-state dependent distributional convex risk measure. See \cref{subsec:DDRM}.}\\
$\varsigma^\fX_{t,x,T}, \varsigma^\fX_{t,x,\infty}$ & \makecell{\small Distributional DRMs. See \eqref{eq:DefsigmatT} and \eqref{eq:DefsigmatInfty}.}\\
$\Psi$ & \makecell{\small Subset of $(\fX,\fA)$. \small See \cref{subsec:Problem}.}\\
$\mu, P$ & \makecell{\small Initial distribution of $X_1$ and transition kernel from $(X_t,A_t)$ to $X_{t+1}$.\\ \small See \cref{asmp:PlaceHolder}.}\\
$X^\fp_t,A^\fp_t,\fX^\fp,\fA^\fp$ & \makecell{\small Counterparts of $X_t,A_t,\fX,\fA$ subject to Markovian policy $\fp$.\\ \small See \cref{subsec:Problem}.}\\
$C_t, \fC$ & \makecell{\small Cost function and abbreviation of $(C_t)_{t\in\bN}$. See \cref{subsec:Problem}.}\\
$O$ & \makecell{\small Function on $\bX$ that is constant $0$.}\\
$H_0, G^\lambda_{t}, H^{\fp}_t, S_t$ & \makecell{\small Operators acting on $\ell^\infty(\bX,\cB(\bX)))$. See \eqref{eq:DefH0}, \eqref{eq:DefG}, \eqref{eq:DefH} and \eqref{eq:DefS}.}\\
$J^\fp_{t,T}, J^*_{t,T}$ & \makecell{\small Value functions of $\fp$ and optimal value functions in finite horizon.\\ See \eqref{eq:DefJtTp} and \eqref{eq:DefJstartT}.}\\
$J^{\fp}_{t,\infty}, J^{*}_{t,\infty}$ & \makecell{\small Value functions of $\fp$ and optimal value functions in infinite horizon.\\ See \eqref{eq:DeftInftyP} and \eqref{eq:DefJStartInfty}.}\\
\hline
\end{tabular}
}
\end{center}
\end{minipage}

\newpage

\bibliographystyle{siamplain}
\bibliography{refs}

\end{document}